\documentclass{amsart}   	

\usepackage{amsmath}
\usepackage{amssymb}
\usepackage{amsfonts}
\usepackage{hyperref}
\usepackage{tikz-cd}
\usepackage{amscd}
\usepackage[all]{xy}
\usepackage{enumerate}

\def\Z{{\mathbb Z}} 
\def\Q{{\mathbb Q}} 
 
\def\C{{\mathbb C}}

\def\bH{{\mathbb H}} 

\def\cC{{\mathcal C}}

\def\cD{{\mathcal D}}
\def\d{{\mathfrak{d}}}

\def\cG{{\mathcal{G}}}
\def\g{{\mathfrak{g}}}
\def\hyp{{\mathrm{hyp}}}
\def\H{{\mathcal H}}

\def\L{{\mathbb{L}}}
\def\M{{\mathcal M}}

\def\cP{{\mathcal P}}
\def\p{{\mathfrak{p}}}

\def\U{{\mathcal U}}
\def\u{{\mathfrak{u}}}

\def\cV{{\mathcal V}}
\def\v{{\mathfrak v}}

\def\Z{{\mathbb{Z}}}

\def\G{{\Gamma}}

\def\Ql{{\Q_\ell}}

\def\ab{\mathrm{ab}}
\def\an{\mathrm{an}}
\def\adj{\mathrm{adj}}

\def\prol{{(\ell)}}

\def\et{\mathrm{\acute{e}t}}

\def\hyp{\mathrm{hyp}}

\def\top{\mathrm{top}}
\def\orb{\mathrm{orb}}

\def\alg{\mathrm{alg}} 

\def\sp{\mathfrak{sp}}

\newcommand{\ad}{\operatorname{ad}}

\newcommand{\Gr}{\operatorname{Gr}}
\newcommand{\Hom}{\operatorname{Hom}}

\newcommand{\Lie}{\operatorname{Lie}}

\newcommand{\Der}{\operatorname{Der}}

\newcommand\im{\operatorname{im}} 
\newcommand\id{\operatorname{id}}

\newcommand\Spec{\operatorname{Spec}}

\newcommand\Sp{\operatorname{Sp}} 

\newtheorem{theorem}{Theorem}[section]
\newtheorem{lemma}[theorem]{Lemma}
\newtheorem{proposition}[theorem]{Proposition}
\newtheorem{corollary}[theorem]{Corollary}
\newtheorem{bigtheorem}{Theorem}

\theoremstyle{definition}
\newtheorem{definition}[theorem]{Definition}
\newtheorem{example}[theorem]{Example}

\theoremstyle{remark}
\newtheorem{remark}[theorem]{Remark}

\begin{document}

\title{On the universal curves with unordered marked points
}
\author{
Ma Luo \and Tatsunari Watanabe
}
\address{School of Mathematical Sciences,  Key Laboratory of MEA (Ministry of Education), Shanghai Key Laboratory of PMMP,  East China Normal University, Shanghai}
\email{mluo@math.ecnu.edu.cn}
\address{Mathematics Department, Embry-Riddle Aeronautical University, Prescott}
\email{watanabt@erau.edu}

\thanks{The first author is supported in part by National Natural Science Foundation of China (No. 12201217), by Shanghai Pilot Program for Basic Research, and by Science and Technology Commission of Shanghai Municipality (No. 22DZ2229014).}

\begin{abstract}
\smallskip
Over any field of characteristic $0$, we prove that the homotopy exact sequence of algebraic fundamental groups for the universal curve with unordered marked points does not split.  The same nonsplitting holds for the universal hyperelliptic curve.  Our approach extends Chen’s topological result to the profinite setting and relies on the use of relative and continuous relative completions to detect the nonexistence of algebraic sections.

\end{abstract}

\maketitle

\setcounter{tocdepth}{1}
\tableofcontents

\section{Introduction}
For nonnegative integers $g$ and $n$ satisfying $2g - 2 + n > 0$ and a field $k$ of characteristic $0$, let $\M_{g,n/k}$ denote the moduli stack of proper, smooth curves of genus $g$ with $n$ ordered distinct marked points over $k$, and let  
\[
\pi_{g,n/k} : \cC_{g,n/k}\to \M_{g,n/k}
\]
be the universal curve over $\M_{g,n/k}$. The symmetric group on $n$ letters, denoted $S_n$, acts on $\M_{g,n/k}$ by permuting the $n$ marked points. We denote the quotient stack $[\M_{g,n/k}/S_n]$ by $\M_{g,[n]/k}$. The $S_n$-action extends to $\cC_{g,n/k}$, and we denote the quotient stack $[\cC_{g,n/k}/S_n]$ by $\cC_{g,[n]/k}$. Thus the universal curve $\pi_{g,n/k}$ induces a family of curves of genus $g$ with $n$ unordered marked points over $\M_{g,[n]/k}$:
\[
\pi_{g,[n]/k}: \cC_{g,[n]/k}\to \M_{g,[n]/k},
\]
which we call the universal curve of genus $g$ with $n$ unordered marked points.

Denote by $\pi_{g,n}:\cC^\an_{g,n}\to \M^\an_{g,n}$ and $\pi_{g,[n]}:\cC^\an_{g,[n]}\to \M^\an_{g,[n]}$ the corresponding complex analytic orbifolds of $\pi_{g,n/\C}$ and $\pi_{g,[n]/\C}$, respectively. In \cite[Thm.~1.3]{chen_uni}, Chen showed that for $n \geq 2$ and $g \geq 2$, the curve $\pi_{g,[n]}$ admits no sections. In this paper, we extend Chen's result to the algebraic fundamental groups of $\pi_{g,[n]/k}$.  

Let $\bar y$ be a geometric point of $\M_{g,[n]/k}$ and $C$ the fiber of $\pi_{g,[n]/k}$ over $\bar y$. Fix a geometric point $\bar x \in C$, viewed also as a geometric point of $\cC_{g,[n]/k}$. Associated to the universal curve $\pi_{g,[n]/k}$, there is a homotopy exact sequence of algebraic fundamental groups:
\begin{equation}\label{homotopy seq for unordered uni curve}
1\to \pi_1^\alg(C, \bar x)\to \pi_1^\alg(\cC_{g,[n]/k}, \bar x)\to \pi_1^\alg(\M_{g,[n]/k}, \bar y)\to 1.
\end{equation}

\begin{bigtheorem}\label{main result for uni curve}
If $g\geq 3$, then the homotopy exact sequence \eqref{homotopy seq for unordered uni curve} does not split.
\end{bigtheorem}

In \cite[Thm.~2.6]{chen_uni}, Chen also proved an analogous result for the hyperelliptic case. Let $\H_{g,n/k}$ be the closed substack of $\M_{g,n/k}$ parametrizing hyperelliptic curves of genus $g$ with $n$ ordered marked points. Denote the quotient stack $[\H_{g,n/k}/S_n]$ by $\H_{g,[n]/k}$; this is a closed substack of $\M_{g,[n]/k}$ parametrizing hyperelliptic curves of genus $g$ with $n$ unordered marked points. Restricting $\pi_{g,n/k}$ to $\H_{g,n/k}$, we obtain the universal hyperelliptic curve
\[
\pi^\hyp_{g,n/k}: \cC_{\H_{g,n/k}}\to \H_{g,n/k}.
\]
Extending the $S_n$-action to $\cC_{\H_{g,n/k}}$ yields a family over $\H_{g,[n]/k}$:
\[
\pi^\hyp_{g,[n]/k}:\cC_{\H_{g,[n]/k}}\to \H_{g,[n]/k},
\]
which we call the universal hyperelliptic curve of genus $g$ with $n$ unordered marked points.

Denote by $\pi^{\hyp}_{g,n}:\cC^\an_{\H_{g,n}}\to \H^\an_{g,n}$ and $\pi^{\hyp}_{g,[n]}:\cC^\an_{\H_{g,[n]}}\to \H^\an_{g,[n]}$ the corresponding complex analytic orbifolds of $\pi^\hyp_{g,n/\C}$ and $\pi^\hyp_{g,[n]/\C}$, respectively. In \cite[Thm.~2.6]{chen_uni}, Chen showed that for $n \geq 2$ and $g \geq 2$, the hyperelliptic universal curve $\pi^{\hyp}_{g,[n]}$ admits no sections.  

We extend Chen's result to the algebraic fundamental groups of $\pi^\hyp_{g,[n]/k}$. Associated to this universal hyperelliptic curve, there is a homotopy exact sequence of algebraic fundamental groups:
\begin{equation}\label{homotopy seq for unordered uni hyp curve}
1\to \pi_1^\alg(C, \bar x)\to \pi_1^\alg(\cC_{\H_{g,[n]/k}}, \bar x)\to \pi_1^\alg(\H_{g,[n]/k}, \bar y)\to 1,
\end{equation}
where $\bar y$ is a geometric point of $\H_{g,[n]/k}$, $C$ is the fiber of $\pi^\hyp_{g,[n]/k}$, and $\bar x$ is a geometric point of $C$.

\begin{bigtheorem}\label{main result for uni hyp curve}
If $g\geq 3$, then the homotopy exact sequence \eqref{homotopy seq for unordered uni hyp curve} does not split.
\end{bigtheorem}

A key tool used in this paper is the relative completion of a discrete group, or, in the case of a profinite group, its continuous relative completion. A detailed account of their constructions and basic properties can be found in \textcolor{black}{\cite{hain_infpretor, hain_hodge_rel, Hain15}}. In particular, for the algebraic fundamental groups of the stacks considered here, their continuous relative completions over \(\Q_\ell\) are canonically isomorphic to the base change of the corresponding relative completions over \(\Q\) to \(\Q_\ell\). 

For example, the relative completion of the orbifold fundamental group 
\(\pi_1^\orb(\M^\an_{g,n})\) is a proalgebraic group over \(\Q\), defined as an extension of \(\Sp(H):=\Sp(H_1(C,\Q))\) by a prounipotent group. Its Lie algebra carries a natural mixed Hodge structure (MHS), and hence admits a weight filtration defined over \(\Q\). Each graded piece of the associated graded Lie algebra of the completion is an \(\Sp(H)\)-representation. Moreover, the universal curve \(\pi_{g,n}\) induces an \(\Sp(H)\)-equivariant graded Lie algebra section of a projection between the pronilpotent Lie algebras of the corresponding completions. 

The \(S_n\)-action together with the \(\Sp(H)\)-module structure on the weight \(-1\) parts of these Lie algebras plays a key role in the proof of Theorem~\ref{main result for uni curve}. In the hyperelliptic case, the relative completion of \(\pi_1^\orb(\H^\an_{g,n})\) admits analogous structures, which are likewise essential for the proof of Theorem~\ref{main result for uni hyp curve}.
\section{Configuration Spaces of Points on a Curve}
Let $C$ be a smooth complex algebraic curve of genus $g$. For a positive integer $n$, the ordered configuration space of $n$ distinct points on $C$ is
$$
F_n(C) := \{(x_1,\dots,x_n) \in C^n \mid x_i \neq x_j \text{ for } i \neq j \}.
$$
The symmetric group $S_n$ acts freely on $F_n(C)$ by permuting the coordinates, and we define the unordered configuration space as the quotient
$$
F_{[n]}(C) := F_n(C)/S_n.
$$

\subsection{Fundamental Groups}
Choose a basepoint configuration $(x_1,\dots,x_n) \in F_n(C)$. The corresponding topological fundamental groups are
$$
\pi_1^\top(F_n(C),(x_1,\dots,x_n)), 
\qquad 
\pi_1^\top(F_{[n]}(C),\{x_1,\dots,x_n\}),
$$
known as the \emph{surface pure braid group} and the \emph{surface braid group} on $n$ strands, respectively.

Since the action of $S_n$ on $F_n(C)$ is free, the covering map
$$
F_n(C) \longrightarrow F_{[n]}(C)
$$
is finite étale with deck group $S_n$. This induces the exact sequence
$$
1 \longrightarrow \pi_1^\top(F_n(C)) \longrightarrow \pi_1^\top(F_{[n]}(C)) \longrightarrow S_n \longrightarrow 1.
$$

\subsection{Unipotent Completions and Mixed Hodge Structures}
Denote by
\[
\cP_{g,n}, \qquad \cP_{g,[n]},
\]
the (rational) unipotent completions over~\(\Q\) of 
\(\pi_1^\top(F_n(C))\) and \(\pi_1^\top(F_{[n]}(C))\), respectively, and set
\[
\p_{g,n} := \Lie(\cP_{g,n}), \qquad 
\p_{g,[n]} := \Lie(\cP_{g,[n]}).
\]
When \(n = 1\), we have \(\pi_1^\top(F_1(C)) = \pi_1^\top(C)\), and we denote \(\p_{g,1}\) by \(\p_g\). 
By work of Morgan and Hain (cf.\ \cite{Hain87, Morgan78}), if $X$ is a smooth complex algebraic variety, then the unipotent completion of $\pi_1^\top(X)$ carries a natural mixed Hodge structure; in particular, both $\p_{g,n}$ and $\p_{g,[n]}$ are pronilpotent Lie algebras over $\Q$ equipped with mixed Hodge structures and weight filtrations defined over $\Q$. The finite étale morphism $F_n(C)\to F_{[n]}(C)$ induces a homomorphism
$$
\cP_{g,n} \longrightarrow \cP_{g,[n]},
$$
and hence a morphism of Lie algebras
$$
\p_{g,n} \longrightarrow \p_{g,[n]},
$$
which is a morphism of mixed Hodge structures and is strictly compatible with the weight filtration. (We do not assert these maps are isomorphisms here.)

\subsection{Continuous Unipotent Completions}\label{cts unipt comp}
For each prime $\ell$, let $\cP_{g,n}^{(\ell)}$ and $\cP_{g,[n]}^{(\ell)}$ denote the continuous (i.e., $\ell$-adic) unipotent completions over $\Q_\ell$ of the profinite groups $\widehat{\pi}_1^\top(F_n(C))$ and $\widehat{\pi}_1^\top(F_{[n]}(C))$, respectively. There are canonical base-change isomorphisms
$$
\p^{(\ell)}_{g,n}:=\Lie(\cP_{g,n}^{(\ell)}) \;\cong\; \p_{g,n} \otimes_\Q \Q_\ell,
\qquad
\p^{(\ell)}_{g, [n]}:=\Lie(\cP_{g,[n]}^{(\ell)}) \;\cong\; \p_{g,[n]} \otimes_\Q \Q_\ell,
$$
so the $\ell$-adic objects are obtained from their rational counterparts by extension of scalars.
\section{Mapping class groups}
Suppose that nonnegative integers $g$ and $n$ satisfy $2g-2 + n >0$. Let \( \Sigma_{g,n} \) denote a fixed reference surface of genus \( g \), namely, a closed, connected, oriented topological surface of genus \( g \) equipped with \( n \) ordered, distinct marked points. The \emph{mapping class group} \( \G_{g,n} \) is defined as the group of isotopy classes of orientation-preserving diffeomorphisms of \( \Sigma_{g,n} \) that fix each marked point individually:
\[
\G_{g,n} := \pi_0(\mathrm{Diff}^+(\Sigma_{g,n})),
\]
where isotopies are also required to fix the marked points pointwise.

The symmetric group \( S_n \) acts on \( \G_{g,n} \) by permuting the labels of the marked points. The resulting quotient
\[
\G_{g,[n]} := \G_{g,n} / S_n
\]
is the mapping class group of genus \( g \) surfaces with \( n \) unordered marked points. This leads to the short exact sequence
\begin{equation}\label{S_n action on the G_{g,n}}
1 \to \G_{g,n} \to \G_{g,[n]} \to S_n \to 1.
\end{equation}

For a group \( G \), denote its profinite completion by \( \widehat{G} \). Since \( \G_{g,n} \) is residually finite, taking the profinite completion of the exact sequence \eqref{S_n action on the G_{g,n}} yields the exact sequence of profinite groups:
\[
1 \to \widehat{\G_{g,n}} \to \widehat{\G_{g,[n]}} \to S_n \to 1.
\]

We define the \emph{hyperelliptic mapping class group} \( \Delta_g \subset \G_g \) as the centralizer of a fixed hyperelliptic involution in the mapping class group of a closed genus \( g \) surface without marked points. Using the natural forgetful homomorphism \( \G_{g,n} \to \G_g \), we define the hyperelliptic mapping class group with \( n \) ordered marked points by the fiber product
\[
\Delta_{g,n} := \Delta_g \times_{\G_g} \G_{g,n},
\]
and similarly, the version with unordered marked points as
\[
\Delta_{g,[n]} := \Delta_g \times_{\G_g} \G_{g,[n]}.
\]
These groups consist of mapping classes that preserve the hyperelliptic structure on the surface after forgetting the marked points.

\subsection{Symplectic Representations and Torelli Groups}\label{symp rep of mcg}

Both \( \G_{g,n} \) and \( \G_{g,[n]} \) admit natural symplectic representations
\[
  \rho_{g,n}: \G_{g,n} \longrightarrow \Sp(H_1(\Sigma_g,\Z)), 
  \qquad 
  \rho_{g,[n]}: \G_{g,[n]} \longrightarrow \Sp(H_1(\Sigma_g,\Z)),
\]
induced by their actions on \( H_1(\Sigma_g,\Z) \).  
It is a basic fact that \( \rho_{g,n} \) is surjective.  
Since the image of \( \rho_{g,[n]} \) contains that of \( \rho_{g,n} \), 
the representation \( \rho_{g,[n]} \) is also surjective.  
The kernel of \( \rho_{g,n} \) is the \emph{Torelli group} \( T_{g,n} \), consisting of 
mapping classes that act trivially on homology.

These representations restrict to the hyperelliptic mapping class groups 
\( \Delta_{g,n} \) and \( \Delta_{g,[n]} \), yielding
\[
  \rho^\hyp_{g,n}: \Delta_{g,n} \longrightarrow \Sp(H_1(\Sigma_g,\Z)), 
  \qquad 
  \rho^\hyp_{g,[n]}: \Delta_{g,[n]} \longrightarrow \Sp(H_1(\Sigma_g,\Z)).
\]
The kernel of \( \rho^\hyp_{g,n} \) defines the \emph{hyperelliptic Torelli group} 
\( T\Delta_{g,n} \), namely the subgroup of \( \Delta_{g,n} \) acting trivially on homology.  
The image of \( \rho^\hyp_g \) contains a finite-index subgroup of 
\( \Sp(H_1(\Sigma_g,\Z)) \) (see A’Campo~\cite{Acampo1979}), and hence the 
representations \( \rho^\hyp_{g,n} \) and \( \rho^\hyp_{g,[n]} \) both have 
Zariski-dense images in \( \Sp(H_1(\Sigma_g,\Q)) \).

Over \( \Q \), the group \( \Sp(H_1(\Sigma_g,\Q)) \) is a connected, reductive 
algebraic group whose finite-dimensional irreducible representations are 
completely classified by their \emph{highest weights}.  
Equivalently, they are classified by \emph{partitions} 
\(\alpha = (\alpha_1 \ge \alpha_2 \ge \cdots \ge \alpha_g \ge 0)\) 
of an integer \( n = \alpha_1 + \cdots + \alpha_g \) into at most \( g \) parts.
Every finite-dimensional rational representation of 
\(\Sp(H_1(\Sigma_g,\Q))\) then decomposes as a direct sum 
of such irreducible components.

\section{Moduli space of curves}
For integers $g,n$ with $2g - 2 + n > 0$, let $\M_{g,n/\Z}$ denote the moduli stack of smooth, proper algebraic curves of genus $g$ with $n$ ordered distinct marked points. The symmetric group $S_n$ acts on $\M_{g,n/\Z}$ by permuting the labels of the marked points. The quotient stack

$$
\M_{g,[n]/\Z} := [\M_{g,n/\Z} / S_n]
$$

    is then the moduli stack of curves with $n$ unordered marked points. In fact, by combining Knudsen’s result on $\M_{g,n/\Z}$ with Zintl’s construction of quotient stacks, one sees that  
$
\M_{g,[n]/\Z}
$
is a smooth Deligne–Mumford stack over $\Z$ whenever $2g - 2 + n > 0$ (cf.\ \cite[Prop.~3.17, Rem.~3.19]{zintl_perm}).

The complex analytification of the moduli stack with ordered marked points is given by  
$$
\M^\an_{g,n}:=\M_{g,n/\C}(\C) \;\simeq\; [\,\mathcal{T}_{g,n} / \G_{g,n}\,],
$$
where $\mathcal{T}_{g,n}$ denotes the Teichmüller space of genus $g$ surfaces with $n$ ordered marked points.  

Similarly, for the moduli stack with unordered marked points, the complex analytification is  
$$
\M^\an_{g,[n]} := \M_{g,[n]/\C}(\C) \;\simeq\; [\,\mathcal{T}_{g,n} / \G_{g,[n]}\,].
$$
Since $\mathcal{T}_{g,n}$ is contractible, the orbifold fundamental groups of these stacks are naturally isomorphic to the corresponding mapping class groups:
$$
\pi_1^{\orb}(\M^\an_{g,n}) \;\simeq\; \G_{g,n}, 
\qquad 
\pi_1^{\orb}(\M^\an_{g,[n]}) \;\simeq\; \G_{g,[n]}.
$$

Moreover, if $k$ is an algebraically closed field of characteristic $0$, the algebraic fundamental groups of the stacks are naturally isomorphic to the profinite completions of the corresponding mapping class groups:
$$
\pi_1^{\alg}(\M_{g,n/k}) \;\simeq\; \widehat{\G_{g,n}}, 
\qquad 
\pi_1^{\alg}(\M_{g,[n]/k}) \;\simeq\; \widehat{\G_{g,[n]}}.
$$
Let $k$ be a field. Consider the universal family
$$
\pi_{g,n/k}: \cC_{g,n/k} \to \M_{g,n/k}
$$
of smooth genus $g$ curves with $n$ marked points over $k$. We define
$$
\cC_{g,[n]/k} := [\,\cC_{g,n/k} / S_n\,],
$$
the quotient of the universal curve $\cC_{g,n/k} \to \M_{g,n/k}$ by the natural action of the symmetric group $S_n$.  

A \emph{geometric point} of $\M_{g,n/k}$ is a smooth, proper $n$--pointed curve defined over an algebraic closure $\bar{k}$,
\[
(C;\, x_1,\dots,x_n), \qquad C/\bar{k} \text{ smooth and proper of genus } g.
\]
The fiber of the universal curve $\cC_{g,n/k} \to \M_{g,n/k}$ over this point is the whole curve $C$. Thus a geometric point of $\cC_{g,n/k}$ lying above $(C;\, x_1,\dots,x_n)$ is
$$
(C;\, x_1,\dots,x_n;\, x_0), \qquad x_0 \in C(\bar{k}).
$$

The group $S_n$ acts by permuting the marked points while leaving the extra point $x_0$ fixed:
$$
(C;\, x_1,\dots,x_n;\, x_0)\;\;\mapsto\;\;(C;\, x_{\tau(1)},\dots,x_{\tau(n)};\, x_0),
\qquad \tau \in S_n.
$$
As observed by Zintl~\cite[Rem.~3.13]{zintl_perm}, the $S_n$–action identifies 
fibers of $\cC_{g,n/k}$ lying over geometric points of $\M_{g,n/k}$ 
that differ only by a permutation of the markings. 
Passing to the quotient yields a natural morphism
$$
\pi_{g ,[n]/k} : \cC_{g,[n]/k} \;\longrightarrow\; \M_{g,[n]/k},
$$
whose fiber over a geometric point 
$$
[(C;\{x_1,\dots,x_n\})] \in \M_{g,[n]/k}(\bar{k})
$$
is the whole curve $C$, regarded together with its set of $n$ unordered marked points.  

Equivalently, an object of $\cC_{g,[n]/k}$ over a scheme $T/k$ consists of an étale $S_n$–torsor $P \to T$ and an $S_n$–equivariant morphism 
$$
P \;\longrightarrow\; \cC_{g,n/k},
$$
corresponding to a family of curves with $n$ ordered disjoint sections over $P$ together with an additional section $x_0$, all compatible with the $S_n$–action.
\subsection{Homotopy exact sequences}
Let $k$ be a field of characteristic $0$, and let $\Omega$ be an algebraically closed field containing $k$. 
  Fix a geometric point $\bar y : \Spec(\Omega) \to \M_{g,n/k}$, and let 
$
C := \pi_{g,n/k}^{-1}(\bar y)
$
denote the geometric fiber of the universal curve over $\bar y$. 
Choose a geometric point $\bar x : \Spec(\Omega) \to C$, and let $\bar x$ also denote its image in $\cC_{g,n/k}$. 
Then there is a homotopy exact sequence of algebraic fundamental groups:
$$
1 \longrightarrow \pi_1^\alg(C, \bar x) 
   \longrightarrow \pi_1^\alg(\cC_{g,n/k}, \bar x) 
   \longrightarrow \pi_1^\alg(\M_{g,n/k}, \bar y) 
   \longrightarrow 1.
$$

An analogous exact sequence holds for the family 
$
\pi_{g,[n]/k}:\cC_{g,[n]/k} \to \M_{g,[n]/k}.
$ 
Let $\bar x \in \cC_{g,[n]}(\bar{k})$ and $\bar y \in \M_{g,[n]}(\bar{k})$ also denote the images of $\bar x$ and $\bar y$, respectively, under the natural projections. 
Then we have another short exact sequence:
$$
1 \longrightarrow \pi_1^\alg(C, \bar x) 
   \longrightarrow \pi_1^\alg(\cC_{g,[n]/k}, \bar x) 
   \longrightarrow \pi_1^\alg(\M_{g,[n]/k}, \bar y) 
   \longrightarrow 1.
$$

The two exact sequences fit into the following commutative diagram:
$$
\begin{tikzcd}
1 \arrow[r] & \pi_1^\alg(C, \bar x) \arrow[r] \arrow[d, "="] 
  & \pi_1^\alg(\cC_{g,n/k}, \bar x) \arrow[r] \arrow[d] 
  & \pi_1^\alg(\M_{g,n/k}, \bar y) \arrow[r] \arrow[d] 
  & 1 \\
1 \arrow[r] & \pi_1^\alg(C, \bar x) \arrow[r] 
  & \pi_1^\alg(\cC_{g,[n]/k}, \bar x) \arrow[r] 
  & \pi_1^\alg
  (\M_{g,[n]/k}, \bar y) \arrow[r] 
  & 1
\end{tikzcd}
$$
where the vertical arrows are induced by the finite étale quotient maps 
$\cC_{g,n/k} \to \cC_{g,[n]/k}$ and $\M_{g,n/k} \to \M_{g,[n]/k}$.

\subsection*{The Hyperelliptic Loci \(\H_{g,n}\) and \(\H_{g,[n]}\)}

Let $k$ be a field of characteristic $0$.  
Let 
\[
\H_{g,n/k} \subset \M_{g,n/k}
\]
denote the closed substack parametrizing smooth, proper genus $g$ hyperelliptic curves with $n$ ordered marked points. 

The stack $\H_{g,n/k}$ is a smooth Deligne--Mumford stack over $k$, provided $2g - 2 + n > 0$. 
The symmetric group $S_n$ acts on $\H_{g,n/k}$ by permuting the labels of the markings, and we define the moduli stack of hyperelliptic curves with unordered markings over $k$ as the quotient stack
\[
\H_{g,[n]/k} := [\H_{g,n/k} / S_n].
\]
The canonical projection $\H_{g,n/k} \to \H_{g,[n]/k}$ is a finite étale morphism induced by the $S_n$–action, and $\H_{g,[n]/k}$ is also a smooth Deligne--Mumford stack over $k$ when $2g - 2 + n > 0$.

Over $\C$, the analytifications
\[
\H^{\an}_{g,n} := \H_{g,n/\C}(\C), 
\qquad
\H^{\an}_{g,[n]} := \H_{g,[n]/\C}(\C)
\]
admit descriptions as orbifold quotients of Teichmüller space.
As in the case of $\M^\an_{g,n}$ and $\M^\an_{g,[n]}$, the orbifold fundamental groups are identified with the hyperelliptic mapping class groups:
\[
\pi_1^{\orb}(\H^{\an}_{g,n}) \;\simeq\; \Delta_{g,n}, 
\qquad
\pi_1^{\orb}(\H^{\an}_{g,[n]}) \;\simeq\; \Delta_{g,[n]}.
\]
For an algebraic closure $\bar k$ of $k$, there are natural isomorphisms
\[
\pi_1^{\alg}(\H_{g,n/\bar k}) \;\cong\; \widehat{\Delta_{g,n}},
\qquad
\pi_1^{\alg}(\H_{g,[n]/\bar k}) \;\cong\; \widehat{\Delta_{g,[n]}}.
\]

Let
\[
\pi^\hyp_{g,n/k} : \cC_{\H_{g,n/k}} \to \H_{g,n/k}
\]
denote the universal hyperelliptic curve with $n$ ordered marked points, and let
\[
\pi^\hyp_{g,[n]/k} : \cC_{\H_{g,[n]/k}} := [\cC_{\H_{g,n/k}} / S_n] \;\longrightarrow\; \H_{g,[n]/k}
\]
be the universal hyperelliptic curve with $n$ unordered markings.  

Over $\C$, their analytifications are given by
\[
\cC^{\an}_{\H_{g,n}} := \cC_{\H_{g,n/\C}}(\C),
\qquad
\cC^{\an}_{\H_{g,[n]}} := \cC_{\H_{g,[n]/\C}}(\C),
\]
which correspond to the universal hyperelliptic curves over $\H^{\an}_{g,n}$ and $\H^{\an}_{g,[n]}$.

Fix a geometric point $\bar y : \Spec(\Omega) \to \H_{g,n/k}$ and let
\[
C := (\pi^\hyp_{g,n/k})^{-1}(\bar y)
\]
be the fiber, with $\bar x \in C(\Omega)$ a chosen geometric point.  
Then there exist short exact sequences of algebraic fundamental groups:
\[
1 \longrightarrow \pi_1^\alg(C, \bar x)
  \longrightarrow \pi_1^\alg(\cC_{\H_{g,n/k}}, \bar x)
  \longrightarrow \pi_1^\alg(\H_{g,n/k}, \bar y)
  \longrightarrow 1,
\]
\[
1 \longrightarrow \pi_1^\alg(C, \bar x)
  \longrightarrow \pi_1^\alg(\cC_{\H_{g,[n]/k}}, \bar x)
  \longrightarrow \pi_1^\alg(\H_{g,[n]/k}, \bar y)
  \longrightarrow 1.
\]

These fit into a commutative diagram:
\[
\begin{tikzcd}
1 \arrow{r} & \pi_1^\alg(C, \bar x) \arrow{r} \arrow[d, "="] &
\pi_1^\alg(\cC_{\H_{g,n/k}}, \bar x) \arrow{r} \arrow[d] &
\pi_1^\alg(\H_{g,n/k}, \bar y) \arrow{r} \arrow[d] & 1 \\
1 \arrow{r} & \pi_1^\alg(C, \bar x) \arrow{r} &
\pi_1^\alg(\cC_{\H_{g,[n]/k}}, \bar x) \arrow{r} &
\pi_1^\alg(\H_{g,[n]/k}, \bar y) \arrow{r} & 1
\end{tikzcd}
\]
where the vertical arrows are induced by the finite étale quotient maps 
$\cC_{\H_{g,n/k}} \to \cC_{\H_{g,[n]/k}}$ and $\H_{g,n/k} \to \H_{g,[n]/k}$.
\section{Relative completion}
In this section, we review definitions and properties of relative completion and its continuous version, more details can be found in \textcolor{black}{\cite{hain_completion,hain_hodge_rel, Hain15}}. Then we apply relative completions to various mapping class groups, deduce enough relevant structures of these completions towards our main theorems.
\subsection{Relative completion and its structure}
Suppose that 
\begin{enumerate}
    \item $\Gamma$ is a discrete group,
    \item $R$ is a reductive group over $F$ ,
    \item $\rho:\Gamma\to R(F)$ is a Zariski dense homomorphism.
\end{enumerate}
The \textit{completion of $\Gamma$ relative to $\rho$}, i.e. the \textit{relative completion of  $\Gamma$}, is a pro-algebraic group $\cG_{/F}$ defined over $F$ and a homomorphism $\widetilde{\rho}:\Gamma\to\cG(F)$ that lifts $\rho$. The group $\cG$ is an extension 
$$1\to\U\to\cG\to R\to1$$
of $R$ by a pro-unipotent group $\U$, which is the unipotent radical of $\cG$. The relative completion has the following universal property: If 
$$1\to U\to G\to R\to 1$$
is an extension of pro-algebraic groups over $F$ such that $U$ is pro-unipotent, and that $\rho:\Gamma\to R(F)$ factors through a homomorphism $\phi:\Gamma\to G(F)$ via $\Gamma\to G(F)\to R(F)$, then there exists a unique homomorphism of pro-algebraic groups $\cG\to G$ that commutes with projections to $R$, such that $\phi$ is the composition $\Gamma\to\cG(F)\to G(F)$.
When $\rho$ is trivial, $\cG=\U$ is the unipotent completion of $\Gamma$.
A generalization of Levi's Theorem ensures that the extension for relative completion 
$$1\to\U\to\cG\to R\to1$$
splits and that any two such splittings are conjugate by an element of $\U$. Every splitting induces an isomorphism
$$\cG\cong R\ltimes\U$$
that commutes with the projections to $R$, where the action of $R$ on $\U$ is determined by the splitting. This action is determined by the action of $R$ on the Lie algebra $\u$ of $\U$ as $\U=\exp\u$. To give a presentation of $\cG$, it suffices to give a presentation of $\u$ in the category of $R$-modules. 

By standard arguments, $\u$ has a minimal presentation of the form
$$\u\cong\L(H_1(\u))^\wedge/(\im\varphi)\quad\text{with}\quad\im\varphi\cong H_2(\u)$$
in the category of $R$-modules, where $\L(V)$ denotes the free Lie algebra generated by the vector space $V$, and $\varphi:H_2(\u)\to[\L(H_1(\u))^\wedge,\L(H_1(\u))^\wedge]$ is an injection such that the composite
$$H_2(\u)\to[\L(H_1(\u))^\wedge,\L(H_1(\u))^\wedge]\to\Lambda^2H_1(\u)$$
is the dual to the cup product $\Lambda^2H^1(\u)\to H^2(\u)$.

\begin{theorem}[{\cite[\S3.2]{Hain15}}]\label{str} 
    For every finite dimensional $R$-module $V$, there is a natural homomorphism 
    $$\Hom_R(H_i(\u),V)\cong(H^i(\u)\otimes V)^R\to H^i(\Gamma,V)$$
    that is an isomorphism when $i=1$  and injective when $i=2$. If every irreducible finite dimensional representation of $R$ is absolutely irreducible, then there is a natural $R$-module isomorphism
    $$H^1(\u)\cong\bigoplus_\alpha H^1(\Gamma,V_\alpha)\otimes V_\alpha^*$$
    and a natural $R$-module injection
    $$H^2(\u)\hookrightarrow\bigoplus_\alpha H^1(\Gamma,V_\alpha)\otimes V_\alpha^*$$
    where $\{V_\alpha\}$ is a set of representatives of the isomorphism classes of irreducible finite dimensional $R$-modules, and $V_\alpha^*$ is the dual of $V_\alpha$.
\end{theorem}

\subsection{Properties of relative completion}

\begin{proposition}[Naturality {\cite[Prop. 3.5]{Hain15}}] 
    For $j=1,2$, let $\rho_j:\Gamma_j\to R_j(F)$ be Zariski dense homomorphisms. Let $\cG_j$ and $\widetilde{\rho}_j:\Gamma_j\to\cG_j(F)$ be the completion of $\Gamma_j$ relative to $\rho_j$. If the diagram
    \begin{center}
        \begin{tikzcd}
            \Gamma_1 \ar[r, "\rho_1"] \ar[d, "\phi_\Gamma"'] & R_1(F) \ar[d, "\phi_R"]\\
            \Gamma_2 \ar[r, "\rho_2"] & R_2(F)
        \end{tikzcd}
    \end{center}
    commutes, then there is a unique homomorphism $\phi_\cG$ such that the following diagram commutes.
    \begin{center}
        \begin{tikzcd}
            \Gamma_1 \ar[r, "\widetilde{\rho}_1"] \ar[d, "\phi_\Gamma"'] \ar[rr, bend left, "\rho_1"] & \cG_1(F) \ar[r] \ar[d, "\phi_\cG"] & R_1(F) \ar[d, "\phi_R"]\\
            \Gamma_2 \ar[r, "\widetilde{\rho}_2"] \ar[rr, bend right, "\rho_2"'] & \cG_2(F) \ar[r] & R_2(F)
        \end{tikzcd}
    \end{center}
\end{proposition}

\begin{proposition}[Right exactness {\cite[Prop. 3.6]{Hain15}}]
    For $j=1,2,3$, let $\rho_j:\Gamma_j\to R_j(F)$ be Zariski dense homomorphisms. Let $\cG_j$ and $\widetilde{\rho}_j:\Gamma_j\to\cG_j(F)$ be the completion of $\Gamma_j$ relative to $\rho_j$. Suppose that one has the diagram 
    \begin{center}
        \begin{tikzcd}
            1 \ar[r] & \Gamma_1 \ar[r]\ar[d] & \Gamma_2 \ar[r]\ar[d] & \Gamma_3 \ar[r]\ar[d] & 1 \\
            1 \ar[r] & R_1 \ar[r] & R_2 \ar[r] & R_3 \ar[r] & 1
        \end{tikzcd}
    \end{center}
    with exact rows, then the corresponding diagram of relative completions 
    \begin{center}
        \begin{tikzcd}
             & \cG_1 \ar[r]\ar[d] & \cG_2 \ar[r]\ar[d] & \cG_3 \ar[r]\ar[d] & 1 \\
            1 \ar[r] & R_1 \ar[r] & R_2 \ar[r] & R_3 \ar[r] & 1
        \end{tikzcd}
    \end{center}
    has right exact top row.
\end{proposition}

\begin{proposition}[Base change {\cite[Cor. 4.14]{hain_completion}}]
    Suppose that $F$  is a field of characteristic zero, and the relative completion $\cG_{/\Q}$ of $\rho:\Gamma\to R(\Q)$ is defined over $\Q$, then there is a natural isomorphism 
    $$\cG_{/F}\xrightarrow{\cong}\cG_{/\Q}\otimes_\Q F.$$
\end{proposition}

We are mostly interested in the case when we base change from $\Q$ to $F=\Ql$. This is closely related to the continuous relative completion which is defined as follows.

\begin{definition}[Continuous relative completion]
Suppose that 
\begin{enumerate}
    \item $\widehat\Gamma$ is a profinite group,
    \item $R$ is a reductive group over $\Q_\ell$,
    \item $\rho^\prol:\widehat\Gamma\to R(\Q_\ell)$ is a \textit{continuous} Zariski dense homomorphism. 
\end{enumerate}
The \textit{continuous completion of $\widehat\Gamma$ relative to $\rho^\prol$} 
(or the \textit{continuous relative completion of $\widehat\Gamma$}) 
is a pro-algebraic group $\cG^\prol$ defined over $\Q_\ell$, together with a 
\textit{continuous} homomorphism 
\[
  \widetilde{\rho}^\prol:\widehat\Gamma \longrightarrow \cG^\prol(\Q_\ell)
\]
that lifts $\rho^\prol$.  
The group $\cG^\prol$ is as an extension of the reductive group $R$ 
by a prounipotent group $\U^\prol$, which is the unipotent radical of $\cG^\prol$:
\[
  1 \longrightarrow \U^\prol \longrightarrow \cG^\prol 
  \longrightarrow R \longrightarrow 1.
\]
\end{definition}
This continuous relative completion satisfies a universal property similar to that of relative completion, but every homomorphism is now required to be \textit{continuous} with respect to the given topology.

Suppose that $\widehat\Gamma$ comes from a discrete group, i.e., it is the profinite completion of a discrete group $\Gamma$. Suppose that $\rho:\Gamma\to R(\Q_\ell)$ is a continuous, Zariski dense homomorphism, so that $\rho^\prol:\widehat\Gamma\to R(\Q_\ell)$ is the continuous extension of $\rho$. The following result is stated in \cite[Thm 6.3]{hain_rational}.

\begin{proposition}
    Suppose $\rho$ and $\rho^\prol$ are given as above. If $\cG^\prol$ and $\tilde\rho:\Gamma\to\cG^\prol(\Q_\ell)$ is the completion of $\Gamma$ relative to $\rho$, then:
    \begin{enumerate}
        \item $\tilde\rho$ is continuous and thus induces a continuous homomorphism $\widetilde{\rho}^\prol:\widehat\Gamma\to\cG^\prol(\Q_\ell)$;
        \item $\cG^\prol$ and $\widetilde{\rho}^\prol$ form the continuous relative completion of $\widehat\Gamma$ with respect to $\rho^\prol$.
    \end{enumerate}
\end{proposition}

\begin{remark}
    This is a generalization of \textcolor{black}{\S\ref{cts unipt comp}}, since $\cG^\prol$ can be obtained from $\cG_\Q$ by base change. The proposition says that base change is compatible with the continuous relative completion. More precisely, the relative completion of the discrete group $\Gamma$ base changed from $\Q$ to $\Ql$ is the same as the continuous relative completion of the profinite group $\widehat{\Gamma}$ over $\Ql$. 
\end{remark}

\subsection{Relative completion of the mapping class groups}
From now on, we will usually omit the subscript $_{/\Q}$ when the group/Lie algebra is defined over $\Q$. In view of the base change property, the results over $\Q$ can be easily extended to the case $F=\Ql$. 

Our primary case of interest is when the discrete group $\Gamma$ is the orbifold fundamental group of the moduli space $\M_{g,n}$, viewed as an analytic orbifold $\M^\an_{g,n}$.  This group is naturally isomorphic to the mapping class group $\Gamma_{g,n}$.  Since the unipotent completion of $\Gamma_{g,n}$ is trivial, it is natural in this setting to consider its relative completion instead.

\begin{example}[Relative completions of $\pi_1^\orb(\M^\an_{g,n})$ and $\pi_1^\orb(\cC^\an_{g,n})$]
For the universal curve 
\[
  \pi_{g,n}\colon \cC^\an_{g,n} \longrightarrow \M^\an_{g,n},
\]
there is a homotopy exact sequence of orbifold fundamental groups
\[
  1 \longrightarrow \pi_1^\top(C)
    \longrightarrow \pi_1^\orb(\cC^\an_{g,n})
    \xrightarrow{\ \pi_{g,n\ast}\ } 
    \pi_1^\orb(\M^\an_{g,n})
    \longrightarrow 1 .
\]
Let $\mathbb{H}_\Q := R^1\pi_{g,n\ast}\Q(1)$ denote the local system over $\M_{g,n}^\an$ whose fiber is
\[
  H := H^1(C,\Q(1)) \cong H_1(C,\Q).
\]
It underlies a variation of Hodge structure (VHS). 
The monodromy action induces a Zariski-dense representation
\[
  \rho_{g,n}\colon \pi_1^\orb(\M^\an_{g,n}) \longrightarrow \Sp(H),
\]
which agrees with the symplectic representation 
$\rho_{g,n}\colon \Gamma_{g,n}\to \Sp(H_1(\Sigma_g,\Z))$
from \S\ref{symp rep of mcg}, once we identify 
$\pi_1^\orb(\M^\an_{g,n}) \cong \Gamma_{g,n}$.

Denote by $\cG_{g,n}$ the relative completion of $\pi_1^\orb(\M^\an_{g,n})$ with respect to $\rho_{g,n}$, and by $\U_{g,n}$ its unipotent radical.  
This yields the exact sequence
\[
  1 \longrightarrow \U_{g,n} \longrightarrow \cG_{g,n} \longrightarrow \Sp(H) \longrightarrow 1.
\]
Since $\pi_{g,n\ast}$ is surjective, the composite
\[
  \pi_1^\orb(\cC^\an_{g,n})
    \xrightarrow{\ \pi_{g,n\ast}\ } \pi_1^\orb(\M^\an_{g,n})
    \xrightarrow{\ \rho_{g,n}\ } \Sp(H)
\]
is also Zariski dense.  
Let $\cG_{\cC_{g,n}}$ be the relative completion of $\pi_1^\orb(\cC^\an_{g,n})$ with respect to this composite, and $\U_{\cC_{g,n}}$ its unipotent radical.  
We denote their Lie algebras by
\[
  \g_{g,n}:=\Lie(\cG_{g,n}),\quad 
  \u_{g,n}:=\Lie(\U_{g,n}),\quad
  \g_{\cC_{g,n}}:=\Lie(\cG_{\cC_{g,n}}),\quad 
  \u_{\cC_{g,n}}:=\Lie(\U_{\cC_{g,n}}).
\]
\end{example}

We may apply the same construction in the case of unordered marked points.
\begin{example}[Relative completions in the unordered case]\label{rel comp in unordered case}
Denote the relative completions of 
$\pi_1^\orb(\M^\an_{g,[n]})$ and $\pi_1^\orb(\cC^\an_{g,[n]})$
by $\cG_{g,[n]}$ and $\cG_{\cC_{g,[n]}}$, and their unipotent radicals by 
$\U_{g,[n]}$ and $\U_{\cC_{g,[n]}}$, respectively.  
These fit into the exact sequences
\[
  \begin{aligned}
  1 &\longrightarrow \U_{g,[n]} 
      \longrightarrow \cG_{g,[n]} 
      \longrightarrow \Sp(H) 
      \longrightarrow 1,\\[4pt]
  1 &\longrightarrow \U_{\cC_{g,[n]}} 
      \longrightarrow \cG_{\cC_{g,[n]}} 
      \longrightarrow \Sp(H) 
      \longrightarrow 1.
  \end{aligned}
\]
We denote the corresponding Lie algebras by
\[
  \g_{g,[n]} := \Lie(\cG_{g,[n]}),\,
  \u_{g,[n]} := \Lie(\U_{g,[n]}),\,
  \g_{\cC_{g,[n]}} := \Lie(\cG_{\cC_{g,[n]}}), \,
  \u_{\cC_{g,[n]}} := \Lie(\U_{\cC_{g,[n]}}).
\]
\end{example}

In the same manner, we can describe the relative completions associated to the hyperelliptic moduli spaces and their universal curves.
\begin{example}[Relative completions in the hyperelliptic cases]
For the universal hyperelliptic curve 
\[
  \pi^\hyp_{g,n}\colon \cC^\an_{\H_{g,n}} \longrightarrow \H^\an_{g,n},
\]
there is a homotopy exact sequence
\[
  1 \longrightarrow \pi_1^\top(C)
    \longrightarrow \pi_1^\orb(\cC^\an_{\H_{g,n}})
    \xrightarrow{\ \pi^\hyp_{g,n\ast}\ } 
    \pi_1^\orb(\H^\an_{g,n})
    \longrightarrow 1 .
\]
Let $\mathbb{H}_\Q := R^1\pi^\hyp_{g,n\ast}\Q(1)$ denote the local system with fiber 
\[
  H := H^1(C,\Q(1)) \cong H_1(C,\Q).
\]
We use the same notation $\mathbb{H}_\Q$ for its pullback of $R^1\pi_{g,n\ast}\Q(1)$ to $\H_{g,n}^\an$.
The monodromy action induces a Zariski-dense representation
\[
  \rho^\hyp_{g,n}\colon \pi_1^\orb(\H^\an_{g,n}) \longrightarrow \Sp(H),
\]
which agrees with the representation 
$\rho^\hyp_{g,n}\colon \Delta_{g,n}\to \Sp(H_1(\Sigma_g,\Z))$ after identifying 
$\pi_1^\orb(\H^\an_{g,n})\cong \Delta_{g,n}$.

Denote by $\cD_{g,n}$ the relative completion of $\pi_1^\orb(\H^\an_{g,n})$ with respect to $\rho^\hyp_{g,n}$ and by $\cV_{g,n}$ its unipotent radical:
\[
  1 \longrightarrow \cV_{g,n} \longrightarrow \cD_{g,n} \longrightarrow \Sp(H) \longrightarrow 1.
\]
The induced composite
\[
  \pi_1^\orb(\cC^\an_{\H_{g,n}})
    \xrightarrow{\ \pi^\hyp_{g,n\ast}\ } \pi_1^\orb(\H^\an_{g,n})
    \xrightarrow{\ \rho^\hyp_{g,n}\ } \Sp(H)
\]
is also Zariski dense, giving the relative completion $\cD_{\cC_{g,n}}$ with unipotent radical $\cV_{\cC_{g,n}}$.  
Their Lie algebras are denoted by
\(\d_{g,n}, \v_{g,n}, \d_{\cC_{g,n}}, \v_{\cC_{g,n}}\).

For the unordered case, we denote the corresponding completions and unipotent radicals by
\[
  \cD_{g,[n]},\ \cV_{g,[n]},\ \cD_{\cC_{g,[n]}},\ \cV_{\cC_{g,[n]}},
\]
fitting into the exact sequences
\[
  1 \to \cV_{g,[n]} \to \cD_{g,[n]} \to \Sp(H) \to 1, 
  \qquad
  1 \to \cV_{\cC_{g,[n]}} \to \cD_{\cC_{g,[n]}} \to \Sp(H) \to 1,
\]
with Lie algebras 
\(\d_{g,[n]}, \v_{g,[n]}, \d_{\cC_{g,[n]}}, \v_{\cC_{g,[n]}}\).
\end{example}

The constructions for the full moduli space and the hyperelliptic locus are compatible by the naturality of relative completion, as summarized below.

\begin{remark}\label{cd}
By the naturality of relative completion, the homomorphisms between the orbifold fundamental groups induced by the closed immersion $\H_{g,n}\hookrightarrow \M_{g,n}$
\[
  \pi_1^\orb(\H^\an_{g,n}) \longrightarrow \pi_1^\orb(\M^\an_{g,n})
  \qquad \text{and} \qquad
  \pi_1^\orb(\cC^\an_{\H_{g,n}}) \longrightarrow \pi_1^\orb(\cC^\an_{g,n})
\]
induce corresponding morphisms between their relative completions.
Consequently, the following diagram with exact rows commutes:
\begin{center}
\begin{tikzcd}
  1 \ar[r] & \U_{g,n} \ar[r] & \cG_{g,n} \ar[r] & \Sp(H) \ar[r] \ar[d, equal] & 1 \\
  1 \ar[r] & \cV_{g,n} \ar[r] \ar[u] & \cD_{g,n} \ar[r] \ar[u] & \Sp(H) \ar[r] & 1
\end{tikzcd}
\end{center}
Passing to Lie algebras yields an analogous commutative diagram
\begin{center}
\begin{tikzcd}
  0 \ar[r] & \u_{g,n} \ar[r] & \g_{g,n} \ar[r] & \sp(H) \ar[r] \ar[d, equal] & 0 \\
  0 \ar[r] & \v_{g,n} \ar[r] \ar[u] & \d_{g,n} \ar[r] \ar[u] & \sp(H) \ar[r] & 0
\end{tikzcd}
\end{center}
with exact rows.
Each Lie algebra in the diagram is equipped with a natural mixed Hodge structure (MHS), 
since the Lie algebras of relative completions carry canonical MHS over $\Q$
(see \cite[Thm.~13.1]{hain_hodge_rel}).  
Under this structure, all induced homomorphisms between the Lie algebras above 
are morphisms of MHSs (see \cite[Thm.~13.12]{hain_hodge_rel}).
\end{remark}



\subsection{Continuous relative completions over \texorpdfstring{$\Q_\ell$}{Ql}}

Let $\bar k$ be an algebraic closure of a field $k$ of characteristic zero.  
Let $\mathbb{H}_{\Q_\ell} := R^1\pi_{g,n\ast}\Q_\ell(1)$, and let 
$H_{\Q_\ell} := H^1_{\mathrm{\acute{e}t}}(C,\Q_\ell(1))$ denote the fiber of $\mathbb{H}_\Ql$ over a geometric point of $\M_{g,n/\bar k}$.  
The monodromy representation
\[
  \rho_{g,n}^{(\ell)}:
  \pi_1^{\alg}(\M_{g,n/\bar k})
  \longrightarrow \Sp(H_{\Q_\ell})
\]
is continuous and Zariski dense.  
We denote by
\[
  1\longrightarrow \U_{g,n}^{(\ell)}\longrightarrow 
  \cG_{g,n}^{(\ell)}\longrightarrow 
  \Sp(H_{\Q_\ell})\longrightarrow 1
\]
the \emph{continuous relative completion} of 
$\pi_1^{\alg}(\M_{g,n/\bar k})$ with respect to $\rho_{g,n}^{(\ell)}$,
and by $\U_{g,n}^{(\ell)}$ its prounipotent radical.  
Analogously, for the universal curve $\cC_{g,n/\bar k}$, we write
\[
  1\longrightarrow \U_{\cC_{g,n}}^{(\ell)}\longrightarrow 
  \cG_{\cC_{g,n}}^{(\ell)}\longrightarrow 
  \Sp(H_{\Q_\ell})\longrightarrow 1.
\]

The same construction applies to the unordered and hyperelliptic cases, yielding
\[
  1\to \U_{g,[n]}^{(\ell)}\to \cG_{g,[n]}^{(\ell)}\to \Sp(H_{\Q_\ell})\to 1,
  \qquad
  1\to \cV_{g,[n]}^{(\ell)}\to \cD_{g,[n]}^{(\ell)}\to \Sp(H_{\Q_\ell})\to 1,
\]
and similarly for their universal curves.  

By the base change theorem of Hain~\cite[Thm.~6.3]{hain_rational},
these continuous completions are obtained from the rational ones by extension of scalars:
\[
  \cG_{g,n}^{(\ell)} \cong \cG_{g,n}\otimes_\Q\Q_\ell,
  \qquad
  \cD_{g,n}^{(\ell)} \cong \cD_{g,n}\otimes_\Q\Q_\ell.
\]

We denote their Lie algebras by
\[
  \g_{g,n}^{(\ell)} := \Lie(\cG_{g,n}^{(\ell)}),\quad 
  \u_{g,n}^{(\ell)} := \Lie(\U_{g,n}^{(\ell)}),\quad
  \d_{g,n}^{(\ell)} := \Lie(\cD_{g,n}^{(\ell)}),\quad
  \v_{g,n}^{(\ell)} := \Lie(\cV_{g,n}^{(\ell)}),
\]
and similarly for the completions associated to the universal curves and unordered cases.

\subsection{Generators of Lie algebras of relative completions}

Recall that the structure of relative completion is determined by the action of the reductive group $R$ on the Lie algebra $\u$ of its unipotent radical. This amounts to give a presentation of $\u$ as $R$-modules. In particular, the presentation of its generator $H_1(\u)$ as $R$-modules can be deduced from Theorem \ref{str}. We now focus on the presentations of generators of various Lie algebras appeared in our examples before.

Note that these Lie algebras are equipped with a natural weight filtration given by Hodge theory (see \cite[Thm.~13.1]{hain_hodge_rel}). This weight filtration satisfies nice exactness properties, so there is no loss of information when replacing a filtered object by its associated graded. More precisely, we will use the following lemma.

\begin{lemma}\label{natural weight splitting}
    The Lie algebras $\u=\p_{g,n}$, $\p_{g,[n]}$, $\u_{g,n}$, $\u_{g,[n]}$, $\u_{\cC_{g,n}}$, $\u_{\cC_{g,[n]}}$, $\v_{g,n}$, $\v_{g,[n]}$, $\v_{\cC_{g,n}}$, $\v_{\cC_{g,[n]}}$ are equipped with a natural weight filtration $W_\bullet$. There is a natural isomorphism 
    $$\u\cong\textcolor{black}{\prod_{m\leq -1}}\Gr^W_m\u$$
    for each of these Lie algebras.
\end{lemma}
The weight filtration $W_\bullet$ on these Lie algebras are essentially the lower central series $L^\bullet$. In particular, these weight filtrations are defined over $\Q$, see \cite[p.28]{hain_relwtfilt}.

\begin{proposition}[{\cite[Props.~2.1, 4.6; Lem.~4.7]{hain_infpretor}}]
If $g\geq 3$, for the Lie algebras \(\u=\p_{g,n}\), \(\u_{g,n}\),
the weight filtration on \(\u\) satisfies
\[
\u = W_{-1}\u,
\qquad
L^l\u = W_{-l}\u,
\]
where \(L^l\u\) denotes the \(l\)-th term of the lower central series. Consequently, there is a canonical isomorphism 
$$H_1(\u)\cong\Gr^W_{-1}\u.$$
\end{proposition}

\begin{remark}
    The  corollary shows that the generators of these Lie algebras is pure of weight $-1$.
\end{remark}
Let $\Lambda^3_0H$ be the representation defined as $(\Lambda^3H)(-1)/(q\wedge H)$, where $q:\Lambda^2H\to \Q(1)$ is the cup product, being viewed as an element in $(\Lambda^2H)(-1)$. It is pure of weight $-1$ and isomorphic to the irreducible symplectic representation corresponding to the partition 1+1+1. For \(v \in \Lambda^3 H\), let \(\bar v\) denote its image in \(\Lambda^3_0 H\).  
Define
\[
\Lambda^3_n H
  := \{(v_1,\ldots,v_n) \in (\Lambda^3 H)^n \mid 
      \bar v_1 = \cdots = \bar v_n\},
\]
which is isomorphic to
\[
\Lambda^3_0 H \oplus \bigoplus_{j=1}^n H_j,
\]
where each \(H_j\) is a copy of \(H\) corresponding to the \(j\)-th component. For brevity, we will often denote \(\bigoplus_{j=1}^n H_j\) by \(H^n\).
\begin{proposition}[{\cite[Prop.~4.6]{hain_infpretor}}]\label{ab of ugn}
If $g\geq 3$, then there is a natural isomorphism
$$
H_1(\u_{g,n})\cong \Lambda^3_nH.
$$
\end{proposition}
For the universal curve case, we have
\begin{corollary}\label{H1 decomp for the univ curve}
If $g\geq 3$, there is an $\Sp(H)$-decomposition
\[
H_1(\u_{\cC_{g,n}})\cong H_0\oplus \Lambda^3_nH\cong \bigoplus_{j=0}^nH_j \oplus \Lambda^3_0H.
\]
\end{corollary}

\begin{proof}
Relative completion applied to the exact sequence
\[
1 \to \pi_1^\top(C)\to\pi_1^\orb(\cC^\an_{g,n})\to \pi_1^\orb(\M^\an_{g,n})\to 1
\]
yields the exact sequence of MHS
\[
0\to\p_g\to \u_{\cC_{g,n}}\to \u_{g,n}\to 0,
\]
where $\p_g =\p_{g,1}$. This gives the exact sequence 
\[
 H_1(\p_g)
  \longrightarrow H_1(\u_{\cC_{g,n}})
  \longrightarrow H_1(\u_{g,n})
  \longrightarrow 0,
\]
which implies that $H_1(\u_{\cC_{g,n}})$ is pure of weight -1.
Since \(\p_g\) is center-free, the adjoint representation 
\(\adj\colon \p_g \to \Der \p_g\) is injective.
Because \(\adj\) factors through \(\u_{\cC_{g,n}}\) and the functor 
\(\Gr^W_{-1}\) is exact, the composition
\[
\Gr^W_{-1}\p_g
  \longrightarrow \Gr^W_{-1}\u_{\cC_{g,n}}
  \longrightarrow \Gr^W_{-1}\Der \p_g
\]
is also injective.
Using the identifications 
\(H_1(\p_g) \cong \Gr^W_{-1}\p_g\) and 
\(H_1(\u_{\cC_{g,n}}) \cong \Gr^W_{-1}H_1(\u_{\cC_{g,n}})\cong\Gr^W_{-1}\u_{\cC_{g,n}}\),
we obtain the commutative diagram
\[
\xymatrix{
\Gr^W_{-1}\p_g \ar[r] \ar[d]^{\cong} &
\Gr^W_{-1}\u_{\cC_{g,n}} \ar[d]^{\cong} \\
H_1(\p_g) \ar[r] &
H_1(\u_{\cC_{g,n}}).
}
\]
Therefore, the sequence
\[
0 \longrightarrow H_1(\p_g)
  \longrightarrow H_1(\u_{\cC_{g,n}})
  \longrightarrow H_1(\u_{g,n})
  \longrightarrow 0
\]
is exact. Consequently, we obtain an \(\Sp(H)\)-decomposition
\[
H_1(\u_{\cC_{g,n}})
  \cong H_1(\p_g) \oplus H_1(\u_{g,n})
  \cong H_0 \oplus \Lambda^3_nH\cong\bigoplus_{j=0}^nH_j \oplus \Lambda^3_0H ,
\]
where \(H_0 := H_1(\p_g)\cong H\). 
\end{proof}
\begin{remark}
Since $H_1(\u_{\cC_{g,n}})$ is pure of weight $-1$, the weight filtration on $\u_{\cC_{g,n}}$ coincides with its lower central series.
\end{remark}
We now turn to the hyperelliptic case.  
A key fact is that \(H_1(\v_g)\) is pure of weight~\(-2\) 
(see~\cite[Prop.~6.9]{wat_rk_hyp_univ}), 
which leads to the following statement.
\begin{proposition}
The weight filtration on \(\v_g\) satisfies
\[
\v_g = W_{-2}\v_g,
\qquad
L^l\v_g = W_{-2l}\v_g,
\]
where \(L^l\v_g\) denotes the \(l\)-th term of the lower central series.
In particular,
\[
H_1(\v_g) \cong \Gr^W_{-2}\v_g.
\]
\end{proposition}
The following lemma describes the resulting canonical \(\Sp(H)\)-decompositions 
of the first homology groups of the corresponding Lie algebras 
\(\v_{g,n}\) and \(\v_{\cC_{g,n}}\).
\begin{lemma}\label{H_1 decomp for rel comp of hyp mcg}
For $g\geq 2$, there are canonical $\Sp(H)$-decompositions:
$$
H_1(\v_{g,n})\cong \Gr^W_{-1}H_1(\v_{g,n})\oplus \Gr^W_{-2}H_1(\v_{g,n})\cong \bigoplus_{j=1}^nH_j\oplus H_1(\v_g)
$$
and 
$$
H_1(\v_{\cC_{g,n}})\cong \Gr^W_{-1}H_1(\v_{\cC_{g,n}})\oplus \Gr^W_{-2}H_1(\v_{\cC_{g,n}})\cong \bigoplus_{j=0}^{n}H_j\oplus H_1(\v_g).
$$
\end{lemma}
\begin{proof}
The exact sequence
$$
1\to \pi_1^\top(F_n(C))\to \pi_1^\orb(\H^\an_{g,n})\to \pi_1^\orb(\H^\an_{g})\to 1
$$
induces the exact sequence of pronilpotent Lie algebras
$$
0\to \p_{g,n}\to \v_{g,n}\to \v_g\to 0,
$$
which in turn yields the exact sequence
$$
H_1(\p_{g,n})\to H_1(\v_{g,n})\to H_1(\v_g)\to 0.
$$
The adjoint action of $\v_{g,n}$ on $\p_{g,n}$ yields the Lie algebra map
$$
\v_{g,n}\to \Der\p_{g,n}.
$$
The center-freeness of $\p_{g,n}$ (see \cite{NTU}) implies that the composition 
$$
\p_{g,n}\to \v_{g,n}\to \Der\p_{g,n}
$$
is an injection.
Since the functor $\Gr^W_\bullet$ is exact, the composition
$$
H_1(\p_{g,n})\cong\Gr^W_{-1}\p_{g,n}\to \Gr^W_{-1}\v_{g,n}\to \Gr^W_{-1}\Der\p_{g,n}
$$
is also injective. Since $\Gr^W_{-1}\v_{g,n}\cong \Gr^W_{-1}H_1(\v_{g,n})$ and the diagram
$$
\xymatrix{
H_1(\p_{g,n})\ar[r]\ar[dr]&H_1(\v_{g,n})\ar[d]\\
&\Gr^W_{-1}H_1(\v_{g,n})
}
$$
commutes, the sequence 
\begin{equation}
0\to H_1(\p_{g,n})\to H_1(\v_{g,n})\to H_1(\v_{g})\to 0
\end{equation} is exact. 
Applying the functor $\Gr^W_\bullet$, we obtain the exact sequence
\begin{equation}
0\to \Gr^W_\bullet H_1(\p_{g,n})\to \Gr^W_\bullet H_1(\v_{g,n})\to \Gr^W_\bullet H_1(\v_{g})\to 0.
\end{equation}
By \textcolor{black}{\cite[Prop.~6.9]{wat_rk_hyp_univ}}, 
$H_1(\v_g)$ is pure of weight $-2$, and so there is a canonical isomorphism 
$$
\Gr^W_{-1}H_1(\v_{g,n})\cong \Gr^W_{-1}H_1(\p_{g,n})=H_1(\p_{g,n}).
$$
Composing with the projection $H_1(\v_{g,n})\to \Gr^W_{-1}H_1(\v_{g,n})$, we get a canonical projection $H_1(\v_{g,n})\to H_1(\p_{g,n})$. This yields a canonical splitting
$$
H_1(\v_{g,n}) \cong \Gr^W_{-1}H_1(\v_{g,n})\oplus \Gr^W_{-2}H_1(\v_{g,n}) \cong H_1(\p_{g,n})\oplus H_1(\v_g)\cong \bigoplus_{j=1}^nH_j\oplus H_1(\v_g).
$$

For the universal case, relative completion applied to the exact sequence
$$
1 \to \pi_1^\top(C)\to \pi_1^\orb(\cC^\an_{\H_{g,n}})\to \pi_1^\orb(\H^\an_{g,n})\to 1
$$
yields the exact sequence of MHS:
$$
0\to \p_g\to \v_{\cC_{g,n}}\to \v_{g,n}\to 0.
$$
By a similar argument as above, 
the center-freeness of~\(\p_g\) implies that the sequence
$$
0\to H_1(\p_g)\to H_1(\v_{\cC_{g,n}})\to H_1(\v_{g,n})\to 0
$$ 
is exact, and this fits into the commutative diagram:
$$
\xymatrix{
0\ar[r]&H_1(\p_g)\ar[r]\ar@{=}[d]&H_1(\v_{\cC_{g,n}})\ar[r]\ar[d]&H_1(\v_{g,n})\ar[r]\ar[d]&0\\
0\ar[r]&H_1(\p_g)\ar[r]&H_1(\v_{g,1})\ar[r]&H_1(\v_g)\ar[r]&0,
}
$$
where the middle map is induced by the projection $\cC^\an_{\H_{g,n}}\to \H^\an_{g,1}$ mapping 
$$(C;\, x_1,\dots,x_n;\, x_0)\mapsto (C; x_0).$$
Since the right square is a pullback square, the canonical splitting 
\(H_1(\v_{g,1}) \cong H_1(\p_g) \oplus H_1(\v_g)\) 
induces a canonical decomposition
\[
H_1(\v_{\cC_{g,n}}) 
  \cong H_1(\p_g) \oplus H_1(\v_{g,n}) 
  \cong \bigoplus_{j=0}^{n} H_j\oplus H_1(\v_g),
\]
where \(H_0 := H_1(\p_g)\).
\end{proof}


By Theorem \ref{str}, the generator $H_1(\v_g)$ of the Lie algebra $\v_g$ can be determined by the cohomology of the discrete group $\pi_1^\orb(\H^\an_g)\cong \Delta_g$ with symplectic representation coefficients. The following results provide some information on $H_1(\v_g)$.

Recall that $H=H^1(C, \Q(1))$ is naturally isomorphic to the standard symplectic representation. Let $\Lambda^2_0H$ be \textcolor{black}{defined as $\Lambda^2H/\langle q\rangle$.} \textcolor{black}{It is isomorphic to }the irreducible symplectic representation corresponding to the partition $1+1$.

\begin{lemma}\label{tanaka's comp}
    $H^1(\pi_1^\orb(\H^\an_g),H)=0$ and $H^1(\pi_1^\orb(\H^\an_g),\Lambda^2_0H)=0$.
\end{lemma}

\begin{proof}
    Tanaka's results \cite[Thm.~1.1 and Thm~.1.3]{Tanaka} state that both the homology groups $H_1(\Delta_g,H_1(\Sigma_g, \Z))$ and $H_1(\Delta_g, \otimes^2H_1(\Sigma_g, \Z))$ are torsion, which implies that $H_1(\Delta_g,\Lambda^2H_1(\Sigma_g, \Z))$ is also torsion. This lemma follows by universal coefficient theorem.
\end{proof}
\textcolor{black}{As an immediate consequence, we have the following result by Theorem \ref{str}.} 
\begin{corollary} \label{H and nichi comp in H_1 of v_g}
    $\Hom_{\Sp}(\textcolor{black}{H_1(\v_g)},H)=0$ and $\Hom_{\Sp}(\textcolor{black}{H_1(\v_g)},\Lambda^2_0H)=0$.
\end{corollary}

\subsection{Characteristic classes with $S_n$-actions}

In \cite{hain-matsumoto}, each tautological section
\(x_j\colon \cC^\an_{g,n}\to\M^\an_{g,n}\) of the universal curve
determines a cohomology class
\(\kappa_j\in H^1(\M_{g,n}^\an,\mathbb{H}_\Q)\).
The construction of these classes is described in \cite[\S4]{hain-matsumoto}, to which we refer for further details. 
These \emph{characteristic classes} form a basis of 
$H^1(\M_{g,n}^\an,\mathbb{H}_\Q)$:
\[
H^1(\M_{g,n}^\an,\mathbb{H}_\Q)
  \cong \bigoplus_{j=1}^n \Q\,\kappa_j.
\]
The symmetric group $S_n$ acts naturally on the $n$ marked points of the universal curve, 
and this induces a corresponding permutation action on the classes 
$\kappa_j$ in $H^1(\M_{g,n}^\an,\mathbb{H}_\Q)$.

There is a natural isomorphism
\[
H^1(\M_{g,n}^\an,\mathbb{H}_\Q)\;\cong\; H^1(\pi_1^\orb(\M^\an_{g,n}), H),
\]
and by Theorem~\ref{str},
\[
H^1(\pi_1^\orb(\M^\an_{g,n}), H)\;\cong\; \Hom_{\Sp}(H_1(\u_{g,n}), H).
\]
By Proposition~\ref{ab of ugn},  we have   $H_1(\u_{g,n})\cong \Lambda^3_n H = \Lambda^3_0 H \oplus H^n$, and 
under this isomorphism, the projection of $H_1(\u_{g,n})$ onto the $j$-th component of $H^n$ 
corresponds to the class $\kappa_j/(2g-2)$ (cf.~\cite[\S12]{hain_rational}).  
Hence, the $S_n$-action on the marked points naturally permutes these projections.

Similarly, in the hyperelliptic case, each tautological section \(x_j\) pulls back to \(\H^\an_{g,n}\) via the closed immersion 
\(\H^\an_{g,n} \hookrightarrow \M^\an_{g,n}\), 
defining a characteristic class 
\(\kappa_j^{\hyp} \in H^1(\H^\an_{g,n}, \mathbb{H}_\Q)\). 
By \cite[Prop.~8.3]{wat_rk_hyp_univ}, there is a natural isomorphism 
\[
H^1(\H^\an_{g,n}, \mathbb{H}_\Q)
  \cong \bigoplus_{j=1}^n \Q\,\kappa_j^{\hyp}.
\]
The $S_n$-action on the marked points also naturally 
permutes the $\kappa_j^{\hyp}$.

On the other hand, one has natural isomorphisms
\[
H^1(\pi_1^\orb(\M^\an_{g,n}), \Lambda^3_0H)
  \;\cong\; H^1(\M_{g,n}^\an,\Lambda^3_0\bH_\Q)
  \;\cong\; \Q\,\nu(\cC_{g,n}),
\]
where $\nu(\cC_{g,n})$ denotes the cohomology class of the universal curve 
$\cC^\an_{g,n}\to\M^\an_{g,n}$, as defined in \cite[\S4]{hain-matsumoto}.  
This class is the pullback of $\nu(\cC_{g,1})$ and does not depend on the choice 
of tautological section (see \cite[Prop.~4.7]{hain-matsumoto}).  
Consequently, $\nu(\cC_{g,n})$ is invariant under the $S_n$-action, and therefore,
\[
H^1(\pi_1^\orb(\M^\an_{g,n}), \Lambda^3_0H)
  = H^1(\pi_1^\orb(\M^\an_{g,n}), \Lambda^3_0H)^{S_n}.
\]
Applying Theorem~\ref{str} again gives
\[
H^1(\pi_1^\orb(\M^\an_{g,n}), \Lambda^3_0H)
  \;\cong\; \Hom_{\Sp}(H_1(\u_{g,n}), \Lambda^3_0H),
\]
and hence,
\[
\Hom_{\Sp}(H_1(\u_{g,n}), \Lambda^3_0H)
  = \Hom_{\Sp}(H_1(\u_{g,n}), \Lambda^3_0H)^{S_n}.
\]

\subsection{\'Etale characteristic classes with $S_n$-actions}

Analogous to the analytic case, one can define \'etale characteristic classes over a field $k$ of characteristic $0$.  
Let 
\[
\mathbb{H}_\Ql := R^1\pi_{g,n\ast}\Ql(1)
\]
be the \'etale local system on $\M_{g,n/k}$ associated to the first \'etale cohomology of the fibers of the universal curve 
$\pi_{g,n}\colon \cC_{g,n/k}\to\M_{g,n/k}$.  
For a geometric point $[C]\in \M_{g,n}(\bar k)$, denote its fiber by 
$H_\Ql := H^1_{\et}(C_{\bar k}, \Ql(1))$.

By the comparison theorem between singular and \'etale cohomology, 
\[
H^1_{\et}(\M_{g,n/\bar k}, \mathbb{H}_\Ql)
  \cong H^1(\M_{g,n}^\an, \mathbb{H}_\Q)\otimes \Ql,
\]
so the classes $\kappa_j$ arising from the tautological sections 
$x_j\colon \M_{g,n/k}\to\cC_{g,n/k}$ form a basis
\[
H^1_{\et}(\M_{g,n/k}, \mathbb{H}_\Ql)
  \cong \bigoplus_{j=1}^n \Ql\,\kappa_j.
\]

Likewise, the classes $\kappa_j^{\hyp}$ defined earlier are restricted from 
$\M_{g,n}$ to the hyperelliptic locus 
$\H_{g,n/\bar k}$ and, by the same comparison theorem, form a basis
\[
H^1_{\et}(\H_{g,n/\bar k}, \mathbb{H}_\Ql)
  \cong \bigoplus_{j=1}^n \Ql\,\kappa_j^{\hyp}.
\]
Furthermore, as in the analytic setting, one has natural identifications
\[
H^1_{\et}(\M_{g,n/\bar k}, \mathbb{H}_\Ql)
  \cong H^1(\pi_1^{\alg}(\M_{g,n/\bar k}), H_\Ql)
  \cong\Hom_{\Sp}(H_1(\u^{\prol}_{g,n}), H_\Ql),
\]
where $H_1(\u^{\prol}_{g,n})\cong \Lambda^3_n H_\Ql
  = \Lambda^3_0H_\Ql\oplus H_\Ql^n$.  
Under this correspondence, the projection onto the $j$-th component of $H_\Ql^n$ represents $\kappa_j/(2g-2)$, and the $S_n$-action permutes these components.

Similarly, for the symplectic representation $\Lambda^3_0H_\Ql$, 
\[
H^1_{\et}(\M_{g,n/\bar k}, \Lambda^3_0\mathbb{H}_\Ql)
  \cong H^1(\pi_1^{\alg}(\M_{g,n/\bar k}), \Lambda^3_0H_\Ql)
  \cong \Ql\, \nu(\cC_{g,n}).
\]
Since $\nu(\cC_{g,n})$ is invariant under the $S_n$-action, we have
\[
H^1_{\et}(\M_{g,n/\bar k}, \Lambda^3_0\mathbb{H}_\Ql)
  = H^1_{\et}(\M_{g,n/\bar k}, \Lambda^3_0\mathbb{H}_\Ql)^{S_n},
\]
or equivalently,
\[
\Hom_{\Sp}(H_1(\u^{\prol}_{g,n}), \Lambda^3_0H_\Ql)
  = \Hom_{\Sp}(H_1(\u^{\prol}_{g,n}), \Lambda^3_0H_\Ql)^{S_n}.
\]

\subsection{\textcolor{black}{Symplectic decompositions of $H_1$ in the unordered marked points case}}
In this section, we decompose 
$H_1(\u_{g,[n]})$, $H_1(\u_{\cC_{g,[n]}})$, $H_1(\v_{g,[n]})$, and $H_1(\v_{\cC_{g,[n]}})$ 
into symplectic representations. For the standard symplectic representation component $H$, we have the following key lemma. 

\begin{lemma}\label{H component in H_1}
For $g\geq 2$, there are isomorphisms
$$
\Hom_{\Sp}(H_1(\u_{g,[n]}), H)\cong \Q\quad \text{ and }\quad \Hom_{\Sp}(H_1(\v_{g,[n]}), H)\cong \Q.
$$
\end{lemma}
\begin{proof}
By Theorem \ref{str} and \textcolor{black}{Example} \ref{rel comp in unordered case}, we have
$$
 H^1(\pi_1^\orb(\M^\an_{g, [n]}),H)\cong \Hom_{\Sp}(H_1(\u_{g,[n]}), H).
$$
From the exact sequence  
\[
\begin{tikzcd}
1 \ar[r] 
  & \pi_1^\orb(\M^\an_{g,n}) \ar[d,"\cong"'] \ar[r] 
  & \pi_1^\orb(\M^\an_{g,[n]}) \ar[d,"\cong"'] \ar[r] 
  & S_n \ar[r] & 1 \\
  & \G_{g,n} 
  & \G_{g,[n]}  & &
\end{tikzcd}
\]
the \emph{Hochschild--Serre five-term exact sequence} yields  
\[
1 \longrightarrow 
H^1(S_n, H^{\Gamma_{g,n}}) 
\longrightarrow 
H^1(\Gamma_{g,[n]}, H)
\longrightarrow 
H^1(\Gamma_{g,n}, H)^{S_n}
\longrightarrow 
H^2(S_n, H^{\Gamma_{g,n}}),
\]
where \(H\) is regarded as a \(\Gamma_{g,[n]}\)-module.
Since both $H^1(S_n,H^{\Gamma_{g,n}})$ and $H^2(S_n,H^{\Gamma_{g,n}})$ vanish, 
$$H^1(\Gamma_{g,[n]},H)\cong H^1(\Gamma_{g,n},H)^{S_n}.$$
Since  
\[
H^1(\pi_1^\orb(\M^\an_{g,n}), H)
  \cong H^1(\M_{g,n}^\an, \mathbb{H}_\Q)
  \cong \bigoplus_{j=1}^n \Q\,\kappa_j,
\]
the \(S_n\)-invariant part is isomorphic to \(\Q\).  
Consequently,
\[
\Hom_{\Sp}\!\left(H_1(\u_{g,[n]}), H\right) \cong \Q.
\]
The same argument also works for $\Hom_{\Sp}(H_1(\v_{g,[n]}), H)\cong \Q$, using the natural isomorphism
$$
H^1(\H^\an_{g,n}, \mathbb{H}_\Q)
  \cong \bigoplus_{j=1}^n \Q\,\kappa_j^{\hyp}.
$$
\end{proof}
The natural $S_n$–action on $H_1(\u_{g,n})$ allows one to determine explicitly the decompositions of 
$H_1(\u_{g,[n]})$ and $H_1(\u_{\cC_{g,[n]}})$.
\begin{lemma}\label{ab of u_{g,[n]}}
For \(g \ge 3\), there is an \(\Sp(H)\)-decomposition
\[
H_1(\u_{g,[n]})
  \cong H_1(\u_g) \oplus \!\bigg(\bigoplus_{j=1}^{n} H_j\!\bigg)_{S_n}
  = \Lambda^3_0 H \oplus H,
\]
which is compatible with the projection
\(H_1(\u_{g,n}) \to H_1(\u_{g,[n]})\)
and with the \(\Sp(H)\)-decomposition
\(H_1(\u_{g,n}) = \Lambda^3_0 H \oplus \bigoplus_{j=1}^{n} H_j\).
Furthermore, there is an \(\Sp(H)\)-decomposition
\[
H_1(\u_{\cC_{g,[n]}}) \cong H_0 \oplus \Lambda^3_0 H \oplus H.
\]

\end{lemma}
\begin{proof}
Consider the commutative diagram:
$$
\xymatrix{
1\ar[r]&\pi_1^\top(F_n(C))\ar[r]\ar[d]&\pi_1^\orb(\M^\an_{g,n})\ar[r]\ar[d]&\pi_1^\orb(\M^\an_{g})\ar[r]\ar@{=}[d]&1\\
1\ar[r]&\pi_1^\top(F_{[n]}(C))\ar[r]&\pi_1^\orb(\M^\an_{g,[n]})\ar[r]&\pi_1^\orb(\M^\an_{g})\ar[r]&1.
}
$$
This diagram induces the following commutative diagram:
$$
\xymatrix{
0\ar[r]&\p_{g,n}\ar[r]\ar[d]&\u_{g,n}\ar[r]\ar[d]&\u_g\ar[r]\ar@{=}[d]&0\\
&\p_{g,[n]}\ar[r]&\u_{g,[n]}\ar[r]&\u_g\ar[r]&0,
}
$$
which in turn yields
$$
\xymatrix{
0\ar[r]&H_1(\p_{g,n})\ar[r]\ar[d]&H_1(\u_{g,n})\ar[r]\ar[d]&H_1(\u_g)\ar[r]\ar@{=}[d]&0\\
&H_1(\p_{g,[n]})\ar[r]&H_1(\u_{g,[n]})\ar[r]&H_1(\u_g)\ar[r]&0,
}
$$
where the rows are exact and the maps are all $\Sp(H)$-equivariant. 
Here we have 
$$
H_1(\p_{g,n})\cong H_1(F_n(C),\Q)\quad \text{ and }\quad H_1(\p_{g,[n]})\cong H_1(F_{[n]}(C),\Q),
$$
and by \cite[Prop.~2.1]{hain_infpretor} 
$$
H_1(F_n(C),\Q)\cong \bigoplus_{j=1}^nH_j,
$$
where each \(H_j\) denotes the copy of \(H\) associated with the \(j\)-th component of the product \(C^n\).
From the exact sequence
\[
1\longrightarrow 
\pi_1^\top(F_n(C))\longrightarrow 
\pi_1^\top(F_{[n]}(C))\longrightarrow 
S_n\longrightarrow 1,
\]
the Hochschild--Serre five-term exact sequence in homology gives
\[
H_2(S_n,\Q)\longrightarrow 
H_1(\pi_1^\top(F_n(C)),\Q)_{S_n}
\longrightarrow 
H_1(\pi_1^\top(F_{[n]}(C)),\Q)
\longrightarrow 
H_1(S_n,\Q)\longrightarrow 0.
\]
Since both end terms vanish, we obtain an isomorphism
\[
H_1(\pi_1^\top(F_{[n]}(C),x),\Q)
  \;\cong\;
H_1(\pi_1^\top(F_n(C)),\Q)_{S_n}.
\]
Passing to the Lie algebras of the unipotent completions,
\[
H_1(\p_{g,[n]})
  \;\cong\;
H_1(\p_{g,n})_{S_n}
  \;\cong\;
\Big(\bigoplus_{j=1}^n H_j\Big)_{S_n}
  \;\cong\;
H.
\]
By Lemma \ref{H component in H_1}, $\Hom_{\Sp(H)}(H_1(\u_{g,[n]}), H) =\Q$, and since $H_1(\u_g)\cong \Lambda^3_0H$, 
the surjection $H_1(\p_{g, [n]})\to \ker(H_1(\u_{g,[n]})\to H_1(\u_g))$ implies that 
$$
\ker(H_1(\u_{g,[n]})\to H_1(\u_g)) \cong H.
$$
Therefore, the sequence
$$
0\to H_1(\p_{g,[n]})\to H_1(\u_{g,[n]})\to H_1(\u_g)\to0
$$
is exact. 
Fix an $\Sp(H)$-decomposition
$$
H_1(\u_{g,n})\cong H_1(\u_g)\oplus H_1(\p_{g,n})\cong \Lambda^3_0H\oplus \bigoplus_{j=1}^nH_j. 
$$
This yields an $\Sp(H)$-decomposition 
$$
H_1(\u_{g,[n]})\cong H_1(\u_g)\oplus H_1(\p_{g,[n]}) \cong \Lambda^3_0{H} \oplus H.
$$
With these decompositions, the projection $H_1(\u_{g,n})\to H_1(\u_{g,[n]})$ can be identified with
the quotient map
$$
\Lambda^3_0H\oplus \bigoplus_{j=1}^nH_j \to (\Lambda^3_0H\oplus \bigoplus_{j=1}^nH_j)_{S_n}=\Lambda^3_0H \oplus H,
$$
where $S_n$ acts trivially on $\Lambda^3_0H$. 

For the universal case, applying relative completion to the exact sequence
\[
1 \longrightarrow \pi_1^\top(C)
  \longrightarrow \pi_1^\orb(\cC^\an_{g,[n]})
  \longrightarrow \pi_1^\orb(\M^\an_{g,[n]})
  \longrightarrow 1
\]
yields an exact sequence of MHS
\[
0 \longrightarrow \p_g
  \longrightarrow \u_{\cC_{g,[n]}}
  \longrightarrow \u_{g,[n]}
  \longrightarrow 0.
\]
As in the proof of Lemma \ref{H1 decomp for the univ curve}, this gives the exact sequence 
\[
 0\to H_1(\p_g)
  \longrightarrow H_1(\u_{\cC_{g,[n]}})
  \longrightarrow H_1(\u_{g,[n]})
  \longrightarrow 0.
\]
Consequently, we obtain an \(\Sp(H)\)-decomposition
\[
H_1(\u_{\cC_{g,[n]}})
  \cong H_1(\p_g) \oplus H_1(\u_{g,[n]})
  \cong H_0 \oplus \Lambda^3_0H \oplus H,
\]
where \(H_0 := H_1(\p_g)\). 
\end{proof}

In the hyperelliptic case, the group 
\(H_1(\v_{g,[n]})\) 
admits a canonical decomposition into its weight~\(-1\) and~\(-2\) components.  
The weight~\(-2\) part is given by \(H_1(\v_g)\), whose explicit presentation remains to be determined.

\begin{lemma}\label{ab of v_{g,[n]}}
For $g \geq 2$, there is a canonical $\Sp(H)$-decomposition of $H_1(\v_{g, [n]})$
$$
H_1(\v_{g, [n]}) \cong H_1(\v_g) \oplus H
$$
that is compatible with the $\Sp(H)$-decomposition $H_1(\v_{g,n}) =H_1(\v_g) \oplus \bigoplus_{j=1}^nH_j$. Furthermore, there is an $\Sp(H)$-decomposition
\[
H_1(\v_{\cC_{g,[n]}}) = H_0 \oplus H_1(\v_g)\oplus H.
\]
\end{lemma}
\begin{proof}
Consider the commutative diagram:
$$
\xymatrix{
1\ar[r]&\pi_1^\top(F_n(C))\ar[r]\ar[d]&\pi_1^\orb(\H^\an_{g,n})\ar[r]\ar[d]&\pi_1^\orb(\H^\an_{g})\ar[r]\ar@{=}[d]&1\\
1\ar[r]&\pi_1^\top(F_{[n]}(C))\ar[r]&\pi_1^\orb(\H^\an_{g,[n]})\ar[r]&\pi_1^\orb(\H^\an_{g})\ar[r]&1.
}
$$
This diagram induces the following commutative diagram:
$$
\xymatrix{
0\ar[r]&\p_{g,n}\ar[r]\ar[d]&\v_{g,n}\ar[r]\ar[d]&\v_g\ar[r]\ar@{=}[d]&0\\
&\p_{g,[n]}\ar[r]&\v_{g,[n]}\ar[r]&\v_g\ar[r]&0,
}
$$
which in turn yields the diagram ($\ast$)
$$
\xymatrix{
0\ar[r]&H_1(\p_{g,n})\ar[r]\ar[d]&H_1(\v_{g,n})\ar[r]\ar[d]&H_1(\v_g)\ar[r]\ar@{=}[d]&0\\
&H_1(\p_{g,[n]})\ar[r]&H_1(\v_{g,[n]})\ar[r]&H_1(\v_g)\ar[r]&0,
}
$$
where the rows are exact and the maps are all $\Sp(H)$-equivariant. 
Applying the functor $\Gr^W_{-1}$ to the diagram $(\ast)$, we obtain the commutative diagram
$$
\xymatrix{
0\ar[r]&\bigoplus_{j=1}^nH_j\ar[r]^-
{\cong}\ar[d]&\Gr^W_{-1}H_1(\v_{g,n})\ar[r]\ar[d]&0\\
&H\ar[r]&\Gr^W_{-1}H_1(\v_{g, [n]})\ar[r]& 0
}
$$
Since \(H_1(\p_{g,[n]})\) is pure of weight \(-1\), 
if \(\Gr^W_{-1}H_1(\v_{g,[n]}) = 0\), then 
\(H_1(\v_{g,[n]}) = W_{-2}H_1(\v_{g,[n]}) 
  = \Gr^W_{-2}H_1(\v_{g,[n]})\),
and we would obtain an \(\Sp(H)\)-isomorphism
\[
H_1(\v_{g,[n]})
  = \Gr^W_{-2}H_1(\v_{g,[n]})
  \cong \Gr^W_{-2}H_1(\v_g)
  = H_1(\v_g).
\]
However, by Lemma~\ref{H component in H_1}, 
\(\Hom_{\Sp(H)}(H_1(\v_{g,[n]}), H) = \Q\),
whereas Lemma~\ref{H and nichi comp in H_1 of v_g} 
shows that this assumption cannot occur.  
Hence, \(\Gr^W_{-1}H_1(\v_{g,[n]}) \neq 0\),
and the induced map 
\[
H \longrightarrow \Gr^W_{-1}H_1(\v_{g,[n]})
\]
is an \(\Sp(H)\)-isomorphism.
Since the diagram
$$
\xymatrix{
H_1(\p_{g, [n]})\ar[r]\ar[dr]&H_1(\v_{g,[n]})\ar[d]\\
&\Gr^W_{-1}H_1(\v_{g, [n]})
}
$$
commutes, it follows that 
$$
0\to H_1(\p_{g,[n]})\to H_1(\v_{g,[n]})\to H_1(\v_g)\to0
$$
is exact. The composition $H_1(\v_{g, [n]})\to\Gr^W_{-1}H_1(\v_{g, [n]})\cong H_1(\p_{g,[n]})$ gives a canonical decomposition 
$$
H_1(\v_{g,[n]})\cong H_1(\v_g)\oplus H_1(\p_{g,[n]}) \cong H_1(\v_g) \oplus H.
$$
By Lemma \ref{H_1 decomp for rel comp of hyp mcg}, there is a canonical decomposition
$$
H_1(\v_{g,n})\cong H_1(\v_g)\oplus H_1(\p_{g,n})\cong H_1(\v_g)\oplus \bigoplus_{j=1}^nH_j. 
$$
Since the diagram
\[
\xymatrix{
0\ar[r] & H_1(\p_{g,n}) \ar[r]^-{\cong}\ar[d] 
  & \Gr^W_{-1}H_1(\v_{g,n}) \ar[r]\ar[d] & 0 \\
0\ar[r] & H_1(\p_{g,[n]}) \ar[r]^
-{\cong} 
  & \Gr^W_{-1}H_1(\v_{g,[n]}) \ar[r] & 0
}
\]
commutes, these decompositions are compatible and 
the projection 
\(H_1(\v_{g,n}) \to H_1(\v_{g,[n]})\)
corresponds to the quotient map
\[
H_1(\v_g) \oplus \bigoplus_{j=1}^n H_j 
  \longrightarrow 
  \big(H_1(\v_g) \oplus \bigoplus_{j=1}^n H_j\big)_{S_n}
  = H_1(\v_g) \oplus H,
\]
where \(S_n\) acts trivially on \(H_1(\v_g)\).

For the universal hyperelliptic case, as in the proof of 
Lemma~\ref{H_1 decomp for rel comp of hyp mcg}, 
there is an exact sequence
\[
0 \longrightarrow H_1(\p_g)
  \longrightarrow H_1(\v_{\cC_{g,[n]}})
  \longrightarrow H_1(\v_{g,[n]})
  \longrightarrow 0.
\]
This yields an \(\Sp(H)\)-decomposition
\[
H_1(\v_{\cC_{g,[n]}})
  \cong H_1(\p_g) \oplus H_1(\v_{g,[n]})
  \cong H_0 \oplus H_1(\v_g) \oplus H,
\]
where \(H_0 := H_1(\p_g)\).
\end{proof}


\section{Sections of $\Gr^W_\bullet\u_{\cC_{g,n}}\to \Gr^W_\bullet\u_{g,n}$}
For simplicity, in this section set $H: = H^1_\et(C, \Ql(1))\cong H^1(C, \Q(1))\otimes \Ql$.
In \cite{hain_rational}, Hain computed the $\mathrm{GSp}(H)$-equivariant Lie algebra sections of 
$$
\Gr^W_\bullet \u^\prol_{\cC_{g,n}}/W_{-3}\to \Gr^W_\bullet \u^\prol_{g,n}/W_{-3}.
$$
There are $n$ sections that are induced by the tautological sections of the universal curve over $\M_{g,n/k}$ and there is an extra section due to the fact that $\Lambda^2\Lambda^3_0H$ does not contain the representation $\Lambda^2_0H$.  
In this section, we will review Hain's result on the weight-preserving Lie algebra sections and show that the extra section cannot be induced by the sections of $\Gr^W_\bullet \u^\prol_{\cC_{g,n}}/W_{-5}\to \Gr^W_\bullet \u^\prol_{g,n}/W_{-5}$, using the result from the hyperelliptic case in \cite{wat_hyp_univ}. 

\begin{proposition}[{\cite[Prop.~10.4]{hain_rational}}]\label{lie algebra sections for g >3}
If $g\geq 4$, the restrictions of the $\Sp(H)$-equivariant weight-preserving graded Lie algebra sections of $\Gr^W_\bullet\u^\prol_{\cC_{g,n}}/W_{-3}\to \Gr^W_\bullet\u^\prol_{g,n}/W_{-3}$ to the $\Gr^W_{-1}$-components are given by
$$
\zeta_i:\Lambda^3_0H\oplus \bigoplus_{j=1}^n H_j\to \Lambda^3_0H \oplus \bigoplus_{j=0}^n H_j,\,\,\,\,(v;x_1,\ldots, x_n)\mapsto (v; x_i, x_1, \ldots, x_n)  
$$ for $i =1, \ldots, n$.
If $g = 3$, in addition to the $n$ sections, the restriction to $\Gr^W_{-1}$ admits a section given by
$$
\zeta_0: (v;x_1,\ldots, x_n)\mapsto (v; 0, x_1, \ldots, x_n).
$$
\end{proposition}

In the second author's paper \cite{wat_hyp_univ}, the weighted completion of the algebraic fundamental group $\pi_1^\alg(\H_{g,n/k})$ was used to study the sections of the universal hyperelliptic curve $\cC_{\H_{g,n/k}}\to \H_{g,n/k}$, where $k$ is a field of characteristic zero such that the image of $\ell$-adic cyclotomic character $G_k\to \Z_\ell^\times$ of the Galois group of $k$ is infinite for some prime number $\ell$. The Lie algebra of the completion admits a natural weight filtration, which plays a key role in determining the study of the sections of the universal hyperelliptic curves. 

In this paper, we use the following analogous result for the Lie algebra of the continuous relative completion of $\pi_1^\alg(\H_{g,n/\C})\cong \widehat{\Delta_{g,n}}$. 

\begin{theorem}[{\cite[Thm.~8.4, Cor.~8.6]{wat_hyp_univ}}]\label{hyp weight -1}
If $g\geq 3$, the restrictions of the $\Sp(H)$-equivariant weight-preserving graded Lie algebra sections of $\Gr^W_\bullet\v^\prol_{\cC_{g,n}}/W_{-5}\to \Gr^W_\bullet\v^\prol_{g,n}/W_{-5}$ to the $\Gr^W_{-1}$-components are given by $\zeta^{\pm}_i$
$$
\zeta^{\pm}_i: \bigoplus_{j=1}^n H_j\to \bigoplus_{j=0}^n H_j,\,\,\,\,(x_1,\ldots, x_n)\mapsto (\pm x_i, x_1, \ldots, x_n).
$$
\end{theorem}
The key facts that allow us to deduce the above result, analogous to the corresponding result via weight filtrations from weighted completion, are as follows. First, the weight filtrations on $\p^\prol_{g,n}$ induced by weighted completion and by Hodge theory both coincide with the lower central series and hence agree. Second, by \cite[Prop.~6.9]{wat_rk_hyp_univ} the abelianization $H_1(\v^\prol_g)$ is pure of weight $-2$, implying that the Lie algebra $\v^\prol_g$ admits only even weights; this is the analogous result for the abelianization of the Lie algebra of the unipotent radical of the weighted completion of $\pi_1^\alg(\H_{g,n/k})$ (see \cite[Prop.~7.7]{wat_hyp_univ}). Finally, by Tanaka’s computation (see Lemma \ref{tanaka's comp}), for $g \geq 2$ we have $\Hom_{\Sp}(H_1(\v_g^\prol), \Lambda^2_0 H) = 0$.

Let $k$ be an algebraically closed field of characteristic 0.
Consider a continuous section $h$ of the natural projection
\[ \pi_1^\alg(\cC_{\H_{g,n}/k}) \to \pi_1^\alg(\H_{g,n/k}). \]
By the universal property of continuous relative completion, this section induces a section $d\tilde{h}$ of the surjection $\v^\prol_{\cC_{g,n}} \twoheadrightarrow \v^\prol_{g,n}$.

\begin{lemma}\label{weight preserved for hyp}
If $g\geq 3$ and $n\geq 0$, the section $dh$ induces a weight-preserving $\Sp(H)$-equivariant graded Lie algebra section of $\Gr^W_\bullet\v^\prol_{\cC_{g,n}}\to \Gr^W_\bullet\v^\prol_{g,n}$.
\end{lemma}
\begin{proof}
Each continuous section $h$ of  
\[
\pi^\hyp:=\pi^\hyp_{g,n/k\ast}:\pi_1^\alg(\cC_{\H_{g,n}/k})\to \pi_1^\alg(\H_{g,n/k})
\]
induces a section $\tilde{h}$ of $\tilde{\pi}^{\hyp, \prol}:\cD^\prol_{\cC_{g,n}}\to \cD^\prol_{g,n}$:
\[
\tilde{h}:\cD^\prol_{g,n}\to \cD^\prol_{\cC_{g,n}}.
\]
We use the same notation $\tilde{h}$ for its restriction to $\cV^\prol_{g,n}$. There is a commutative diagram
\[
\xymatrix{
1\ar[r]&\cV^\prol_{\cC_{g,n}}\ar[r]\ar[d]&\cD^\prol_{\cC_{g,n}}\ar[r]\ar[d]&\Sp(H)\ar[r]\ar@{=}[d]&1\\
1\ar[r]&\cV^\prol_{g,n}\ar[r]\ar@/_/[u]_{\tilde{h}}&\cD^\prol_{g,n}\ar[r]\ar@/_/[u]_{\tilde{h}}&\Sp(H)\ar[r]&1.
}
\]
The abelianization $\tilde{h}^\ab$ of $\tilde{h}$ is $\Sp(H)$-equivariant, and so is the section  $d\tilde{h}^\ab$ of 
\[
(d\tilde{\pi}^{\hyp,\prol})^\ab: H_1(\v^\prol_{\cC_{g,n}})\to H_1(\v^\prol_{g,n}).
\]

Consider the exact sequence induced by $\pi^\hyp$:
\[
0\to \p_g^\prol\to \v^\prol_{\cC_{g,n}}\to\v^\prol_{g,n}\to 0.
\]
Since $\p_g^\prol$ is center-free \textcolor{black}{\cite{NTU}}, the adjoint action
\[
\p_g^\prol\to \v^\prol_{\cC_{g,n}}\overset{\ad}{\longrightarrow} \Der\p_g^\prol
\]
is injective. The exactness of the functor $\Gr^W_\bullet$ implies that  
\[
\Gr^W_{-1}\p_g^\prol = H_1(\p_g^\prol) \longrightarrow \Gr^W_{-1}\Der\p_g^\prol
\]
is also injective. As this map factors through
\[
H_1(\v^\prol_{\cC_{g,n}})\to \Gr^W_{-1}\Der \p_g^\prol,
\]
we obtain the short exact sequence
\begin{equation}\label{ab eq of v}
0\to H_1(\p_g^\prol)\to H_1(\v^\prol_{\cC_{g,n}})\to H_1(\v^\prol_{g,n})\to 0.
\end{equation}

By \cite[Prop.~6.9]{wat_rk_hyp_univ}, $H_1(\v^\prol_g)$ is pure of weight $-2$. By Lemma \ref{H_1 decomp for rel comp of hyp mcg}, we have canonical decompositions over $\Ql$: 

\[
H_1(\v^\prol_{\cC_{g,n}})\cong\Gr^W_{-1}H_1(\v^\prol_{\cC_{g,n}})\oplus\Gr^W_{-2}H_1(\v^\prol_{\cC_{g,n}}) \cong H^{n+1} \oplus H_1(\v^\prol_g),
\]
and
\[
H_1(\v^\prol_{g,n}) \cong \Gr^W_{-1}H_1(\v^\prol_{g,n})\oplus \Gr^W_{-2}H_1(\v^\prol_{g,n})\cong H^n \oplus H_1(\v^\prol_g).
\]
Thus, \eqref{ab eq of v} becomes
\[
0\to H_0\to H^{n+1} \oplus H_1(\v^\prol_g)\to H^{n} \oplus H_1(\v^\prol_g)\to 0.
\]

Since $d\tilde{h}^\ab$ is a section of $(d\tilde{\pi}^{\hyp,\prol})^\ab$, it cannot map the $H^n$ component into $H_1(\v_g)$.  
By Lemma \ref{tanaka's comp},
there is no $\Sp(H)$-equivariant map $H_1(\v^\prol_g) \to H$. Hence the restriction of $d\tilde{h}^\ab$ to $H_1(\v^\prol_g)\subset H_1(\v^\prol_{g,n})$ does not land in $H^{n+1}$. It follows that $d\tilde{h}^\ab$ is weight-preserving.

\textcolor{black}{From Lemma \ref{natural weight splitting}, we obtain the following natural isomorphisms over $\Ql$}
\[
\v^\prol_{\cC_{g,n}}\cong \prod_{m\leq -1} \Gr^W_m\v^\prol_{\cC_{g,n}},
\quad
\v^\prol_{g,n}\cong \prod_{m\leq -1} \Gr^W_m\v^\prol_{g,n},
\]
fitting into the commutative diagram
\[
\xymatrix{
\v^\prol_{\cC_{g,n}}\ar[r]^-\cong\ar[d]^{d\tilde{\pi}^{\hyp, \prol}}&\prod_{m\leq -1} \Gr^W_m\v^\prol_{\cC_{g,n}}\ar[d]^{\Gr^W d\tilde{\pi}^{\hyp, \prol}}\\
\v^\prol_{g,n}\ar[r]^-\cong&\prod_{m\leq -1} \Gr^W_m\v^\prol_{g,n}.
}
\]

The section $d\tilde{h}$ induces an $\Sp(H)$-equivariant graded Lie algebra section
\[
\Gr^L_\bullet d\tilde{h} : \Gr^L_\bullet \prod_{m\leq -1} \Gr^W_m\v^\prol_{g,n} \longrightarrow \Gr^L_\bullet \prod_{m\leq -1} \Gr^W_m\v^\prol_{\cC_{g,n}}.
\]
For each bracket length $s$, the map
\[
\Gr^L_s d\tilde{h}: \Gr^L_s \prod_{m\leq -1} \Gr^W_m\v^\prol_{g,n} \longrightarrow \Gr^L_s \prod_{m\leq -1} \Gr^W_m\v^\prol_{\cC_{g,n}}
\]
is induced by
\[
\otimes^s H_1(\v^\prol_{g,n}) \longrightarrow \otimes^s H_1(\v^\prol_{\cC_{g,n}}),
\]
which itself is induced by $d\tilde{h}^\ab$. Since $d\tilde{h}^\ab$ is weight-preserving, so is $\Gr^L_s d\tilde{h}$.
Because $d\tilde{h}^\ab$ is $\Sp(H)$-equivariant, each $\Gr^L_s d\tilde{h}$ is likewise weight-preserving and $\Sp(H)$-equivariant.  
Thus, the graded Lie algebra map $\Gr^L_\bullet d\tilde{h}$ is $\Sp(H)$-equivariant and preserves the weight filtration. Observe that
$$
\Gr^W_\bullet\v^\prol_{g,n} =\Gr^L_\bullet \prod_{m\leq -1} \Gr^W_m\v^\prol_{g,n}\,\,\,\,\text{ and }\,\,\,\,\Gr^W_\bullet\v^\prol_{\cC_{g,n}} =\Gr^L_\bullet \prod_{m\leq -1} \Gr^W_m\v^\prol_{\cC_{g,n}}.
$$
We may therefore regard it as an $\Sp(H)$-equivariant graded Lie algebra section 
\[
\Gr^W_\bullet d\tilde{h}:  \Gr^W_\bullet\v^\prol_{g,n} \to \Gr^W_\bullet \v^\prol_{\cC_{g,n}}
\]
of $\Gr^W_\bullet d\tilde{\pi}^{\hyp, \prol}$.
\end{proof}
\begin{proposition} \label{lie algebra sections for g > 2 in w = -1} 
Let $s$ be a continuous section of $\pi_1^\alg(\cC_{g,n/k})\to \pi_1^\alg(\M_{g,n/k})$. If $g\geq 3$, then the induced section $\Gr^W_{-1}d\tilde{s}$ of $\Gr^W_{-1}\u^\prol_{\cC_{g,n}}\to\Gr^W_{-1}\u^\prol_{g,n}$ is equal to $\zeta_i$ for some $i=1,\ldots,n$.
\end{proposition}
\begin{proof}
Let $s$ be a continuous section of $\pi_1^\alg(\cC_{g,n/k})\to \pi_1^\alg(\M_{g,n/k})$. It induces a Lie algebra section $d\tilde{s}$ of $\u^\prol_{\cC_{g,n}}\to\u^\prol_{g,n}$. Since the weight filtrations on $\u^\prol_{\cC_{g,n}}$ and $\u^\prol_{g,n}$ coincide with their lower central series, $d\tilde{s}$ preserves their weight filtrations. Thus, the section $d\tilde{s}$ yields an $\Sp(H)$-equivariant graded Lie algebra section $\Gr^W_\bullet d\tilde{s}$ of $\Gr^W_\bullet\u^\prol_{\cC_{g,n}}\to \Gr^W_\bullet\u^\prol_{g,n}$.
 When $g\geq 4$, the map $\Gr^W_{-1}d\tilde{s}$ on the $\Gr^W_{-1}$-components is equal to $\zeta_i$ for some $i =1,\ldots,n$. 

 For the case when $g=3$, consider the section $h$ of $\pi_1^\alg(\cC_{\H_{g,n}/k}) \to \pi_1^\alg(\H_{g,n/k})$ obtained by pulling back $s$ along $\pi_1^\alg(\H_{g,n/k})\to \pi_1^\alg(\M_{g,n/k})$. The section $h$ yields a section $d\tilde{h}$ of $\v^\prol_{\cC_{g,n}}\to \v^\prol_{g,n}$. By Lemma \ref{weight preserved for hyp}, $d\tilde{h}$ induces a weight-preserving $\Sp(H)$-equivariant graded Lie algebra section $\Gr^W_\bullet d\tilde{h}$ of $\Gr^W_\bullet\v^\prol_{\cC_{g,n}}\to\Gr^W_\bullet\v^\prol_{g,n}$. In weight $-1$, there is the commutative diagram:
 \[
\xymatrix{
\Gr^W_{-1}\v^\prol_{\cC_{g,n}}= \bigoplus_{j=0}^nH_j\ar[r]\ar[d] &\ar@/^1pc/[l]^{\Gr^W_{-1}d\tilde h}\Gr^W_{-1}\v^\prol_{g,n} =\bigoplus_{j=1}^nH_j\ar[d]\\
\Gr^W_{-1}\u^\prol_{\cC_{g,n}} = \Lambda^3_0H \oplus\bigoplus_{j=0}^nH_j\ar[r] &\ar@/^1pc/[l]^{\Gr^W_{-1}d\tilde s}\Gr^W_{-1}\u^\prol_{g,n} = \Lambda^3_0H\oplus \bigoplus_{j=1}^nH_j,
}
 \]
 where the vertical maps are induced by the closed immersion $\H_{g/k}\hookrightarrow \M_{g/k}$.
 By Theorem \ref{hyp weight -1}, $\Gr^W_{-1}d\tilde{h} = \zeta^{\pm}_i$ for some $i=1,\ldots, n$. 
 Since $\Gr^W_{-1}d\tilde{h}$ is the restriction of $\Gr^W_{-1}d\tilde{s}$ to $\bigoplus_{j=1}^n H_j$, $\Gr^W_{-1}d\tilde{s}$ cannot be equal to $\zeta_0$. Therefore, it is one of the $\zeta_i$ when $g=3$ as well.

\end{proof}
\section{Proofs of the main theorems}


Let $k$ be a field of characteristic $0$, and let $\bar{k}$ denote an algebraic closure of $k$ in an algebraically closed field $\Omega$.  
Suppose that $g \geq 3$. Since a continuous section of $\pi_{g,[n]/k\ast}$ induces one of $\pi_{g,[n]/\bar{k}\ast}$, 
it suffices to prove the main theorems over $\bar{k}$.  
Let $\ell$ be a prime number, and set
$
H := H^1_{\et}(C, \mathbb{Q}_{\ell}(1)).
$
\subsection{Proof of Theorem 1}
 Suppose that the natural projection 
$$\pi_{g,[n]/\bar k\ast}:\pi_1^\alg(\cC_{g,[n]/\bar{k}}, \bar{x}) \to \pi_1^\alg(\M_{g,[n]/\bar{k}}, \bar{y})
$$ admits a continuous section $s$.
 Then $s$ induces an $\Sp(H)$-equivariant graded Lie algebra section
\[
\Gr^W_\bullet d\tilde{s}: \Gr^W_\bullet \u^\prol_{g,[n]} \to \Gr^W_\bullet \u^\prol_{\cC_{g,[n]}}.
\]

Pulling back $s$ along the map
\[
\pi_1^\alg(\M_{g,n/\bar{k}}, \bar{y}) \to \pi_1^\alg(\M_{g,[n]/\bar{k}}, \bar{y}),
\]
we obtain a section
\[
t: \pi_1^\alg(\M_{g,n/\bar{k}}, \bar{y}) \to \pi_1^\alg(\cC_{g,n/\bar{k}}, \bar{x}),
\]
which induces an $\Sp(H)$-equivariant graded Lie algebra section
\[
\Gr^W_\bullet d\tilde{t}: \Gr^W_\bullet \u^\prol_{g,n} \to \Gr^W_\bullet \u^\prol_{\cC_{g,n}}.
\]
By Lemma \ref{ab of u_{g,[n]}}, $H_1(\u_{\cC_{g,[n]}})$ is pure of weight -1, and hence $\Gr^W_{-1}\u^\prol_{\cC_{g,[n]}} \cong H_1(\u^\prol_{\cC_{g,[n]}})$. Now, fix an $\Sp(H)$-decomposition  
\[
H_1(\u^\prol_{\cC_{g,[n]}})\cong H_0\oplus \Lambda^3_0H \oplus H.
\]
Since the diagram
\[
\xymatrix{
H_1(\u^\prol_{\cC_{g,n}})\ar[r]\ar[d]_{\Gr^W_{-1}\mathrm{p}_{\cC}}&
H_1(\u^\prol_{g,n})\ar[d]^{\Gr^W_{-1}\mathrm{p}_{\M}}\\
H_1(\u^\prol_{\cC_{g,[n]}})\ar[r]&
H_1(\u^\prol_{g,[n]})
}
\]
is a pullback square, where the vertical maps are induced by the finite étale morphisms 
\(\cC_{g,n/\bar k}\to \cC_{g,[n]/\bar k}\) and 
\(\M_{g,n/\bar k}\to \M_{g,[n]/\bar k}\), 
the decomposition 
\[
H_1(\u^\prol_{\cC_{g,[n]}})
  \cong H_0 \oplus \Lambda^3_0 H \oplus H
\]
induces an \(\Sp(H)\)-decomposition
\[
H_1(\u^\prol_{\cC_{g,n}})
  \cong H_0 \oplus \Lambda^3_0 H \oplus \bigoplus_{j=1}^n H_j,
\]
under which the projection 
\(\Gr^W_{-1}\mathrm{p}_{\cC}\) 
is identified with
\[
\Gr^W_{-1}\mathrm{p}_{\cC} = (\id_{H_0},\, \Gr^W_{-1}\mathrm{p}_{\M}).
\]
Hence, in weight $-1$, there is a commutative diagram:
$$
\xymatrix{
H_0 \oplus \Lambda^3_0 H \oplus \bigoplus_{j=1}^nH_j \ar[r]\ar[d]^{\Gr^W_{-1}\mathrm{p}_\cC}&\ar@/^1pc/[l]^{\Gr^W_{-1}d\tilde t} \Lambda^3_0 H\oplus \bigoplus_{j=1}^nH_j\ar[d]^{\Gr^W_{-1}\mathrm{p}_\M}\\
H_0 \oplus  \Lambda^3_0 H\oplus H\ar[r] & \ar@/^1pc/[l]^{\Gr^W_{-1}d\tilde s} \Lambda^3_0 H\oplus H.
}
$$
By Proposition \ref{lie algebra sections for g > 2 in w = -1}, $\Gr^W_{-1}d\tilde t$ is equal to $\zeta_i$ for some $i$. Without loss of generality, we may assume $i =1$. Denote by $\mathrm{pr}_0$ the projection 
$$\Gr^W_{-1}\u^\prol_{\cC_{g,[n]}}\cong H_1(\u^\prol_{\cC_{g,[n]}}) = H_0  \oplus \Lambda^3_0 H \oplus H \to H_0.
$$
For any $x\in H$, consider $\bar x_1 :=(0; x,0, \ldots,0)\in\Gr^W_{-1}\u^\prol_{g,n}\cong H_1(\u^\prol_{g,n})=\Lambda^3_0H\oplus \bigoplus_{j=1}^n H_j$. Then we have
\begin{align*}
\mathrm{pr}_0\circ \Gr^W_{-1}d\tilde s\circ \Gr^W_{-1}\mathrm{p}_\M(\bar x_1) &= \mathrm{pr}_0\circ \Gr^W_{-1}\mathrm{p}_\cC\circ \Gr^W_{-1}d\tilde t (\bar x_1)\\
& = \mathrm{pr}_0\circ \Gr^W_{-1}\mathrm{p}_\cC\circ \zeta_1 (\bar x_1)\\
& = x.
\end{align*}
Since $\Gr^W_{-1}\mathrm{p}_\M(\bar x_1) = (0; x)$ in $\Lambda^3_0H \oplus H$, 
the restriction of $\Gr^W_{-1}d\tilde s$ to $H$ followed by $\mathrm{pr}_0$ is given by
$$
\mathrm{pr}_0\circ\Gr^W_{-1}ds: x\mapsto x\quad \text{ for }x\in H.
$$
On the other hand, consider $\bar x_2 := (0; 0, x, 0\ldots, 0)$ in $\Lambda^3_0H\oplus \bigoplus_{j=1}^n H_j$ with $x\not = 0 \in H$ being in the second copy of $H$ in $\bigoplus_{j=1}^n H_j$.
Then we have
\begin{align*}
x = \mathrm{pr}_0\circ \Gr^W_{-1}d\tilde s\circ \Gr^W_{-1}\mathrm{p}_\M(\bar x_2) &= \mathrm{pr}_0\circ \Gr^W_{-1}\mathrm{p}_\cC\circ \Gr^W_{-1}d\tilde t (\bar x_2)\\
& = \mathrm{pr}_0\circ \Gr^W_{-1}\mathrm{p}_\cC\circ \zeta_1 (\bar x_2)\\
& = 0,
\end{align*}
which is a contradiction. Therefore, $s$ does not exist. \qed

\subsection{Proof of Theorem 2}
Consider the universal hyperelliptic curves 
$$\pi^\hyp_{g,n/\bar k}:\cC_{\H_{g, n}/\bar k} \to \H_{g, n/\bar k}\quad \text{ and }\quad \pi^\hyp_{g, [n]/\bar k}:\cC_{\H_{g ,[n]}/\bar k}\to \H_{g, [n]/\bar k}.$$ 
Suppose that the projection $\pi^\hyp_{g, [n]/\bar k\ast} :\pi_1(\cC_{\H_{g, [n]}/\bar k}, \bar x)\to \pi_1(\H_{g, [n]/\bar k}, \bar y)$ admits a continuous section $h$. The commutative diagram
$$
\xymatrix{
\pi_1(\cC_{\H_{g,n}/\bar k}, \bar x)\ar[r]^{\pi^\hyp_{g,n/\bar k\ast}}\ar[d]&\pi_1(\H_{g,n/\bar k},\bar y)\ar[d]\\
\pi_1(\cC_{\H_{g, [n]}/\bar k}, \bar x)\ar[r]^{\pi^\hyp_{g, [n]/\bar k\ast}}& \pi_1(\H_{g, [n]/\bar k}, \bar y)
}
$$
is a pullback square, and so the section $h$ induces a continuous section $z$ of $\pi^\hyp_{g,n/\bar k\ast}$. The section $z$ induces a Lie algebra section $d\tilde z$ of $\v^\prol_{\cC_{g,n}}\to \v^\prol_{g,n}$.  By Lemma \ref{weight preserved for hyp}, $d\tilde z$ induces an $\Sp(H)$-equivariant weight-preserving graded Lie algebra section $\Gr^W_\bullet d\tilde z$ of the projection $\Gr^W_\bullet d\tilde{\pi}^{\hyp,\prol}_{g,n}: \Gr^W_\bullet \v^\prol_{\cC_{g,n}}\to \Gr^W_\bullet\v^\prol_{g,n}$. 
By Lemma~\ref{ab of v_{g,[n]}}, we fix an \(\Sp(H)\)-decomposition
\[
H_1(\v^\prol_{\cC_{g,[n]}})\cong H_0 \oplus H_1(\v^\prol_g) \oplus H,
\]
which, by Lemma~\ref{H_1 decomp for rel comp of hyp mcg}, induces the \(\Sp(H)\)-decomposition
\[
H_1(\v^\prol_{\cC_{g,n}})\cong H_0 \oplus H_1(\v^\prol_g) \oplus \bigoplus_{j=1}^n H_j.
\]
It follows that
\[
\Gr^W_{-1}\v^\prol_{\cC_{g,[n]}}
  \cong \Gr^W_{-1}H_1(\v^\prol_{\cC_{g,[n]}})
  = H_0 \oplus H,
\]and
\[\Gr^W_{-1}\v^\prol_{\cC_{g,n}}
  \cong \Gr^W_{-1}H_1(\v^\prol_{\cC_{g,n}})
  = H_0 \oplus \bigoplus_{j=1}^n H_j.
\]
Moreover, by Lemma~\ref{H_1 decomp for rel comp of hyp mcg},
\[
\Gr^W_{-1}\v^\prol_{g,n}
  \cong \Gr^W_{-1}H_1(\v^\prol_{g,n})
  = \bigoplus_{j=1}^n H_j,
\]
and by Lemma~\ref{ab of v_{g,[n]}}
\[
\Gr^W_{-1}\v^\prol_{g,[n]}
  \cong \Gr^W_{-1}H_1(\v^\prol_{g,[n]})
  = H.
\]
As in the proof of Theorem~1, under these identifications we obtain, in weight~\(-1\), the following commutative diagram:
\[
\xymatrix{
H_0 \oplus \bigoplus_{j=1}^n H_j 
\ar[r]\ar[d]^{\Gr^W_{-1}\mathrm{p}_{\cC_\H}} 
  & \bigoplus_{j=1}^n H_j 
      \ar@/_1pc/[l]_{\Gr^W_{-1}d\tilde z}
      \ar[d]^{\Gr^W_{-1}\mathrm{p}_\H} \\
H_0 \oplus H 
  \ar[r] 
  & H 
      \ar@/^1pc/[l]^{\Gr^W_{-1}d\tilde h}
}
\]
where the vertical maps are induced by the finite étale morphisms 
\(\cC_{\H_{g,n}/\bar k}\to \cC_{\H_{g,[n]}/\bar k}\) and 
\(\H_{g,n/\bar k}\to \H_{g,[n]/\bar k}\). 
Under the above identifications, the projection 
\(\Gr^W_{-1}\mathrm{p}_{\cC_\H}\) is given by
\[
\Gr^W_{-1}\mathrm{p}_{\cC_\H} = (\id_{H_0},\, \Gr^W_{-1}\mathrm{p}_\H).
\]
By Theorem \ref{hyp weight -1}, $\Gr^W_{-1}d\tilde z$ is equal to $\zeta_i^\pm$ for some $i$. With loss of generality, we may assume $ i =1$. Then an argument similar to the proof of Theorem 1 shows that 
$$
\mathrm{pr}_0\circ \Gr^W_{-1}d\tilde h (x) = \pm x \quad \text{ for } x\in H.
$$
On the other hand, for $\bar x_2 =(0, x,0\ldots,0)$ with $x\not =0\in H$ being in the second copy of $H$ in $\bigoplus_{j=1}^nH_j$, we have
\begin{align*}
\pm x = \mathrm{pr}_0\circ \Gr^W_{-1}d\tilde h\circ \Gr^W_{-1}\mathrm{p}_\H(\bar x_2) &= \mathrm{pr}_0\circ \Gr^W_{-1}\mathrm{p}_{\cC_\H}\circ \Gr^W_{-1}d\tilde z (\bar x_2)\\
& = \mathrm{pr}_0\circ \Gr^W_{-1}\mathrm{p}_{\cC_{\H}}\circ \zeta^\pm_1 (\bar x_2)\\
& = 0,
\end{align*}
which is a contradiction. Therefore, $\pi^\hyp_{g, [n]/\bar k\ast}$ does not admit any continuous section. \qed

\section{Notation and Glossary}

Throughout, $k$ denotes a field of characteristic $0$, and $\bar{k}$ its algebraic closure.
All stacks are Deligne–Mumford. For a complex stack $X/\C$, $X^{\an}$ is its analytification.
For a group $G$, $\widehat{G}$ denotes its profinite completion.

\subsection*{Moduli stacks and universal curves}
\begin{align*}
  &\M_{g,n/k}
    && \text{moduli of smooth proper genus-$g$}\\
  &&& \text{curves with $n$ ordered marked points},\\
  &\M_{g,[n]/k}:=[\M_{g,n/k}/S_n]
    && \text{moduli with $n$ unordered points},\\
  &\H_{g,n/k}\subset \M_{g,n/k},\ 
    \H_{g,[n]/k}:=[\H_{g,n/k}/S_n]
    && \text{hyperelliptic loci},\\
  &\cC_{g,n/k}\to \M_{g,n/k},\ 
    \cC_{g,[n]/k}\to \M_{g,[n]/k}
    && \text{universal curves},\\
  &\cC_{\H_{g,n}/k}\to \H_{g,n/k},\ 
    \cC_{\H_{g,[n]}/k}\to \H_{g,[n]/k}
    && \text{universal hyperelliptic curves.}
\end{align*}
\subsection*{Mapping class groups}
\begin{align*}
  &\G_{g,n} && \text{mapping class group, $n$ ordered marked points},\\
  &\G_{g,[n]}:=\G_{g,n}/S_n && \text{unordered version},\\
  &\Delta_g && \text{hyperelliptic mapping class group},\\
  &\Delta_{g,n},\ \Delta_{g,[n]} && \text{hyperelliptic versions with ordered/unordered points},\\
  &T_{g,n},\ T\Delta_{g,n} && \text{Torelli and hyperelliptic Torelli groups}.
\end{align*}

\subsection*{Relative completions (over \texorpdfstring{$\Q$}{Q}) and Lie algebras}
Each group, denoted by $\cG$, $\cD$, or their variants, is the relative completion (over $\Q$) 
of the corresponding orbifold fundamental group with respect to the standard symplectic representation.  
Their unipotent radicals are denoted by $\U$, $\cV$, and their Lie algebras by $\g$, $\d$, $\u$, and $\v$, respectively.
These objects are defined over $\Q$.
\[
\begin{array}{ll}
\cG_{g,n},\ \U_{g,n},\ \g_{g,n},\ \u_{g,n} 
  &:\ \text{completion of }\pi_1^\orb(\M^\an_{g,n})\cong\G_{g,n},\\[2pt]
\cG_{g,[n]},\ \U_{g,[n]},\ \g_{g,[n]},\ \u_{g,[n]} 
  &:\ \text{completion of }\pi_1^\orb(\M^\an_{g,[n]})\cong\G_{g,[n]},\\[2pt]
\cD_{g,n},\ \cV_{g,n},\ \d_{g,n},\ \v_{g,n} 
  &:\ \text{completion of }\pi_1^\orb(\H^\an_{g,n})\cong\Delta_{g,n},\\[2pt]
\cD_{g,[n]},\ \cV_{g,[n]},\ \d_{g,[n]},\ \v_{g,[n]} 
  &:\ \text{completion of }\pi_1^\orb(\H^\an_{g,[n]})\cong\Delta_{g,[n]},\\[2pt]
\cG_{\cC_{g,n}},\ \U_{\cC_{g,n}},\ \g_{\cC_{g,n}},\ \u_{\cC_{g,n}} 
  &:\ \text{completion of }\pi_1^{\orb}(\cC_{g,n}^{\an}),\\[2pt]
\cG_{\cC_{g,[n]}},\ \U_{\cC_{g,[n]}},\ \g_{\cC_{g,[n]}},\ \u_{\cC_{g,[n]}} 
  &:\ \text{completion of }\pi_1^{\orb}(\cC_{g,[n]}^{\an}),\\[2pt]
\cD_{\cC_{g,n}},\ \cV_{\cC_{g,n}},\ \d_{\cC_{g,n}},\ \v_{\H_{g,n}} 
  &:\ \text{completion of }\pi_1^{\orb}(\cC_{\H_{g,n}}^{\an}),\\[2pt]
\cD_{\cC_{g,[n]}},\ \cV_{\cC_{g,[n]}},\ \d_{\cC_{g,[n]}},\ \v_{\cC_{g,[n]}} 
  &:\ \text{completion of }\pi_1^{\orb}(\cC_{\H_{g,[n]}}^{\an}).
\end{array}
\]

\subsection*{Continuous relative completions}

In the \'etale setting, one considers the continuous relative completions of 
algebraic fundamental groups with respect to the standard symplectic representation over~$\Ql$.  
These completions are the continuous analogues of those defined in the analytic setting over~$\Q$, 
and they are related by the natural comparison isomorphisms
\[
  \g^\prol_{g,n} \cong \g_{g,n}\otimes_\Q \Ql,
  \qquad
  \u^\prol_{g,n} \cong \u_{g,n}\otimes_\Q \Ql,
\]
and similarly for all other completions considered below.

Each group denoted by $\cG^\prol$, $\cD^\prol$, or their variants is the continuous relative completion 
of the corresponding \'etale fundamental group with respect to the standard symplectic representation.  
Their unipotent radicals are denoted by $\U^\prol$, $\cV^\prol$, and their Lie algebras by 
$\g^\prol$, $\d^\prol$, $\u^\prol$, and $\v^\prol$, respectively.

Objects used in the continuous case (all over $\Ql$):
\[
\begin{array}{ll}
\cG^\prol_{g,n},\ \U^\prol_{g,n},\ \g^\prol_{g,n},\ \u^\prol_{g,n} 
  &:\ \text{completion of }\pi_1^{\alg}(\M_{g,n/\bar k})\cong\widehat{\G_{g,n}},\\[2pt]
\cG^\prol_{g,[n]},\ \U^\prol_{g,[n]},\ \g^\prol_{g,[n]},\ \u^\prol_{g,[n]} 
      &:\ \text{completion of }\pi_1^{\alg}(\M_{g,[n]/\bar k})\cong\widehat{\G_{g,[n]}},\\[2pt]
\cD^\prol_{g,n},\ \cV^\prol_{g,n},\ \d^\prol_{g,n},\ \v^\prol_{g,n} 
  &:\ \text{completion of }\pi_1^{\alg}(\H_{g,n/\bar k})\cong\widehat{\Delta_{g,n}},\\[2pt]
\cD^\prol_{g,[n]},\ \cV^\prol_{g,[n]},\ \d^\prol_{g,[n]},\ \v^\prol_{g,[n]} 
  &:\ \text{completion of }\pi_1^{\alg}(\H_{g,[n]/\bar k})\cong\widehat{\Delta_{g,[n]}},\\[2pt]
\cG^\prol_{\cC_{g,n}},\ \U^\prol_{\cC_{g,n}},\ \g^\prol_{\cC_{g,n}},\ \u^\prol_{\cC_{g,n}} 
  &:\ \text{completion of }\pi_1^{\alg}(\cC_{g,n/\bar k}),\\[2pt]
\cG^\prol_{\cC_{g,[n]}},\ \U^\prol_{\cC_{g,[n]}},\ \g^\prol_{\cC_{g,[n]}},\ \u^\prol_{\cC_{g,[n]}} 
  &:\ \text{completion of }\pi_1^{\alg}(\cC_{g,[n]/\bar k}),\\[2pt]
\cD^\prol_{\cC_{g,n}},\ \cV^\prol_{\cC_{g,n}},\ \d^\prol_{\cC_{g,n}},\ \v^\prol_{\cC_{g,n}} 
  &:\ \text{completion of }\pi_1^{\alg}(\cC_{\H_{g,n}/\bar k}),\\[2pt]
\cD^\prol_{\cC_{g,[n]}},\ \cV^\prol_{\cC_{g,[n]}},\ \d^\prol_{\cC_{g,[n]}},\ \v^\prol_{\cC_{g,[n]}} 
  &:\ \text{completion of }\pi_1^{\alg}(\cC_{\H_{g,[n]}/\bar k}).
\end{array}
\]
\subsection*{Configuration Lie algebras}
\begin{align*}
  \p_{g,n}
    &:= \Lie\!\left(\text{unipotent completion of }\pi_1^\top(F_n(C))\right),\\
    \p_g: =\p_{g,1}&:= \Lie\!\left(\text{unipotent completion of }\pi_1^\top(C)\right),\\
  \p_{g,[n]}
    &:= \Lie\!\left(\text{unipotent completion of }\pi_1^\top(F_{[n]}(C))\right).
\end{align*}
and their $\Q_\ell$-forms
\[
  \p_{g,n}^{(\ell)}\cong \p_{g,n}\otimes_{\Q}\Q_\ell,\qquad
  \p_{g,[n]}^{(\ell)}\cong \p_{g,[n]}\otimes_{\Q}\Q_\ell.
\]

\bibliographystyle{alpha}   
\bibliography{references}

\end{document}